\documentclass{dmgt}

\usepackage{mathtools}
\usepackage{algorithm}
\usepackage{algpseudocode}
\usepackage{comment}

\newauthor{%
P\v remysl Holub}{%
P. Holub}{%
Department of Mathematics\\
University of West Bohemia\\
Pilsen, Czech Republic}[%
holubpre@kma.zcu.cz]

\newauthor{%
Martin Kop\v riva}{%
M. Kop\v riva}{%
Department of Mathematics\\
University of West Bohemia\\
Pilsen, Czech Republic}[%
mkopriva@kma.zcu.cz]

\title{$L(3,2,1)$-labelings of three classes of $4$-valent circulants}


\classnbr{05C12, 05C15, 05C78}

\newcommand{\diam}{\operatorname{diam}}
\newcommand{\dist}{\operatorname{dist}}

\newtheorem{proposition}[theorem]{Proposition}
\newtheorem{observ}[theorem]{Observation}
\newtheorem{conjecture}[theorem]{Conjecture}
\newtheorem{corollary}[theorem]{Corollary}

\begin{document}

\begin{abstract}
An $L(3,2,1)$-labeling of a graph $G$ is an assignment $f$ of nonnegative integers to vertices such that $\vert f(x)-f(y)\vert > 3-\mbox{dist}_G(x,y)$ for every pair $x,y$ of vertices of $G$, where $\mbox{dist}_G(x,y)$ denotes the distance between $x$ and $y$ in $G$. The minimum span (i.e., the difference between the largest and the smallest value) among all $L(3,2,1)$-labelings of $G$ is denoted by $\lambda_{(3,2,1)}(G)$. In this paper, we study $L(3,2,1)$-labelings of three classes of circulant graphs. Namely, we investigate $\lambda_{(3,2,1)}$ of circulant graphs $C_n(1,t)$, where $t\in\{3,4,5\}$ and $n$ is the order of the graph. 
This paper is a continuation of a recent publications of V. Bianco and T. Calamoneri who studied the square of cycles, i.e., circulant graphs $C_n(1,2)$.
\end{abstract}

\section{Introduction}

The Frequency (or Channel) assignment problem (FAP) has been a motivation for various types of labelings and colourings of graphs (see e.g. \cite{Hale}). Among others, a well-known $L(2,1)$-labeling of graphs was introduced by Griggs and Yeh \cite{GY}. A generalization of this, a problem of $L(p,q)$-labeling of graphs
arose: Given a graph $G$, label vertices of $G$ with nonnegative integers in such a way that labels of adjacent vertices must differ by at least $p$ while vertices at distance $2$ must differ by at least $q$. Naturally, the question is to minimize the span, i.e. the difference between the maximum and the minimum value used for such a labeling of $G$. In the sense of FAP, vertices at distance 1 represent transmitters which are "very close" to each other while vertices at distance 2 represent transmitters which are just "close". The minimum span over all such labelings of $G$ is called the labeling number and is denoted by $\lambda_{(p,q)}(G)$. In the last 30 years, various results on variations of the $L(p,q)$-labeling of graphs have been published, see e.g. survey papers \cite{Cal survey,Yeh survey} and references therein.

Chartrand et al. \cite{CEZH} generalized the concept of an $L(2,1)$-labeling to the concept of a radio labeling in the following way. For a positive integer $k$, a $k$-radio labeling of $G$ is a function  $f:V(G)\rightarrow \mathbb{N}\cup\{0\}$ such that $\vert f(x)-f(y)\vert>k-\dist_G(x,y)$  for each pair of vertices $x,y$ of $G$. Analogously, we ask for such a labeling with a minimum span, which is called the $k$-radio labeling number and is denoted by $rl_k(G)$. The concept of the $k$-radio labeling appears to be difficult in general, only some results for special classes of graphs have been published during last 20 years (see e.g. \cite{CEHT, KSV22} and references therein). A natural extension of the notion of $L(2,1)$-labeling of graphs is a labeling which is based on constraints not only for labels of adjacent vertices and vertices at distance two apart, but also for vertices which are at distance three apart. In this sense, transmitters in the FAP are considered to be "very close", "close" and "relatively close". Recently, some results on $k$-radio labeling for $k=3$ have been presented, see e.g. \cite{Cal,CKL11,clipperton,DGN17,ZH24}. 
Note that, for $k=1$, a $k$-radio labeling becomes the proper vertex colouring, and for $k=2$, we obtain the $L(2,1)$-labeling. In this direction, a $3$-radio labeling is usually called an $L(3,2,1)$-labeling.

In this paper we consider simple undirected graphs only. For definitions and notations not defined here we refer to \cite{BM}. Let $G$ be a graph with vertex set $V(G)$ and edge set $E(G)$. The {\em complement} of $G$, denoted by $\bar G$, is the graph with $V(\bar G)=V(G)$ in which vertices $x$ and $y$ are adjacent whenever $x$ and $y$ are not adjacent in $G$. A {\em subgraph} of $G$ induced by a set $X\subset V(G)$ will be denoted by ${\langle X \rangle}_G$. For $x,y\in V(G)$, $\dist_G(x,y)$ denotes the {\em distance} between $x$ and $y$ in $G$, i.e. the length of a shortest path between $x$ and $y$. The {\em diameter} of $G$, denoted by $\diam(G)$, is the largest distance between any two vertices of $G$. By symbol $P_n$ we mean a path on $n$ vertices. The {\em Cartesian product} of two graphs $G_1=(V(G_1), E(G_1))$ and $G_2=(V(G_2), E(G_2))$, denoted by $G_1 \square G_2$, is the graph $G=(V(G), E(G))$, where $V(G)=V(G_1)\times V(G_2)$ and two vertices $(x_1, y_1)$ and $(x_2,y_2)$ are adjacent if and only if $x_1=x_2$ and $y_1y_2 \in E(G_2)$ or $y_1=y_2$ and $x_1x_2 \in E(G_1)$. By symbol $[k]$, where $k \in \mathbb{N}$, we mean the set $\{1,2,\dots,k\}$. For a set $A$ of nonnegative integers, let $\bar A$ denote the complement of $A$ in $\mathbb{N}$, i.e., a set of all nonnegative integers which do not belong to $A$.

An {\em $L(3,2,1)$-labeling} of $G$ is a function $f \colon V(G) \to \mathbb{N} \cup \{0\}$ such that, for each pair of vertices $x, y \in V(G)$, $\vert f(x)-f(y)\vert > 3-\dist_G(x,y)$. Hence labels of neighbouring vertices must differ by at least $3$, labels of vertices at distance $2$ must differ by at least two, and labels of vertices at distance $3$ must be different.
The {\em span} of an $L(3,2,1)$-labeling of $G$ is the difference between the largest and the smallest value. The minimum span among all $L(3,2,1)$-labelings of $G$ is denoted by $\lambda_{(3,2,1)}(G)$. Since the smallest used value is $0,$ the largest one is $\lambda_{(3,2,1)}(G)$ – the {\em labeling number}. We say that the labeling of $G$ is {\em optimal}, if there is no labeling of $G$ with a smaller span. The notation $f(x_i)=s_i$ means that vertex $x_i$ has label $s_i$. 

The following statement will be used in some proofs through this paper.

\begin{theorem} \cite{CKL11}
\label{podgraf}
Let $H$ be a subgraph of a graph $G$. Then $\lambda_{(3,2,1)}(H) \leq \lambda_{(3,2,1)}(G)$.
\end{theorem}

In this paper we also use the following basic result on $\lambda_{(3,2,1)}$ of cycles.



\begin{theorem} \cite{clipperton}
\label{chromaticke_cislo_kruznic}
Let $n \in \mathbb{N}, n \geq 3$, and $C_n$ be a cycle on $n$ vertices. Then
\begin{equation*}
\lambda_{(3,2,1)}(C_n) =
\,
\left\{ \begin{array}{ll}
6 \quad &\mbox{if} \; n=3, \\
7 \quad &\mbox{if} \; n \; \mbox{is even},\\
8 \quad &\mbox{if} \; n \; \mbox{is odd and if } n \neq 3,7,\\
9 \quad &\mbox{if} \; n=7.\\
\end{array}
\right.
\
\end{equation*}
\end{theorem}



Circulant graphs (circulants) can be equivalently defined in various ways. We will use the following definition.
Let $n \geq 3$ be an integer and let $S \subseteq \{1,2,\dots,\lfloor n/2 \rfloor\}$.  
The \emph{circulant graph} $C_n(S)$ is the graph with vertex set $V = \{u_1,u_2,\dots,u_n\}$,
where two vertices $u_i$ and $u_j$ are adjacent if and only if
\begin{equation*}
\vert i - j\vert  \equiv s \pmod{n} \quad \text{or} \quad \vert i - j\vert \equiv n - s \pmod{n}
\end{equation*}
for some $s \in S$ (for a simplification, we consider $u_n$ instead of $u_0$ by the fact that $n\equiv 0 \pmod{n}$). Equivalently, for each vertex $u_i\in V(C_n(S))$, $u_i$ is adjacent to the vertices with indices
\begin{equation*}
i \pm s \pmod{n}, \quad \text{for all } s \in S.
\end{equation*} 

Clearly, if $\vert S\vert = \lfloor{\frac n2}\rfloor$, then $C_n(S)$ is a complete graph, and if $\vert S\vert =1$, then $C_n(S)$ is a cycle or a union of vertex disjoint cycles. 
For simplicity, we abbreviate the notation $C_n(\{s_1,s_2,\dots, s_k)\}$ to $C_n(s_1,s_2,\dots, s_k)$.
In the rest of the paper we will consider circulant graphs as connected, if not stated otherwise. Some examples of circulant graphs are depicted in Fig. \ref{FIG circ intro}.

\begin{figure}[ht]
  \centering
  \includegraphics[width=0.8\textwidth]{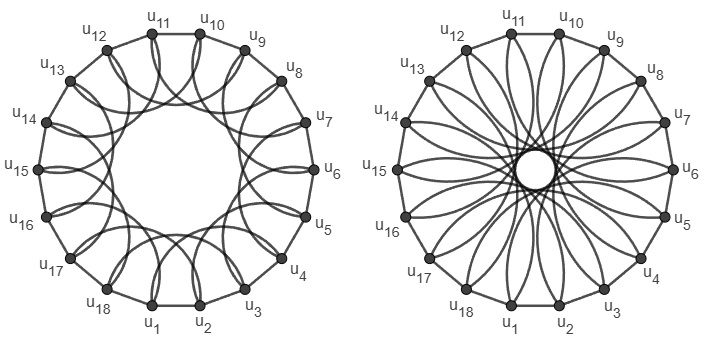}
  \caption{The circulant graphs $C_{18}(1,3)$ (left) and $C_{18}(1,5)$ (right).}
  \label{FIG circ intro}
\end{figure}




A closely related class to circulants is a class of distance graphs. For a set of positive integers $D$, a {\em distance graph} $G(\mathbb{Z},D)$ is a graph with vertex set $\mathbb{Z}$ in which two vertices $a,b$ are adjacent whenever $\vert a-b\vert \in D$. A lot of results concerning various types of colourings and labelings of circulant and distance graphs has been published in the last decades (see e.g. \cite{CEHT,TaoGu}). For simplicity, we will use the notation
$G(D)=G(\mathbb{Z},D)$, e.g., for $D=\{k,t\}$ we will write $G(k,t)$ instead of $G(\mathbb{Z},\{k,t\})$.

A {\em pattern} of length $k$ is a sequence of $k$ values $f(u_1), \dots, f(u_k)$ corresponding to vertices $u_1, \dots, u_k$. A labeling of a circulant graph $G$ with a pattern of length $k$ is a labeling $f$ of $G$ such that $f(u_{i+k})=f(u_i)$ for each $i=1,\dots,\vert V(G)\vert$ where indices are taken modulo $n$. Analogously, for a distance graph $G$, $f(i+k)=f(i)$ for every $i\in\mathbb{Z}$. In other words, we periodically repeat the pattern $f(u_1), \dots, f(u_k)$ along the whole $G$. In this context we say that the pattern forms a {\em feasible} labeling.

We also use the following famous Sylvester's theorem.
\
\begin{theorem}{\cite{sylvester}}
\label{sylvester}
Let $a,b$ be two positive integers such that $a$ and $b$ are coprime, and let $n \in \mathbb{N}$. Then the equation $ax+by=n$, where $x$ and $y$ are nonnegative integers, have a solution $(x,y)$ whenever $n \geq (a-1)(b-1)$.
\end{theorem}

The main goal of this paper is an investigation of $3$-radio labeling number of some specific $4$-valent circulants. The class of circulant graphs is an intensively studied class of graphs thanks to their interesting properties. On one hand, circulants are known as models of double-loop networks and are of interest in the network theory. In the graph theory, circulants are highly symmetric graphs and also Cayley graphs of cyclic groups. Their structure, various graph parameters and properties have been investigated in last couple of decades, including chromatic parameters and labeling numbers. Recently, a series of papers on various $L(p,q)$-labelings of circulants has been published (see \cite{MB1, MB2, MB3}). 
In \cite{Cal} and later on in \cite{Cal2}, the $\lambda_{(3,2,1)}$ of the square of cycles was studied. Note that, the square of a cycle $C_n$ is equivalent to a circulant graph $C_n(1,2)$. Specifically, they showed the following statement.

\begin{theorem}
For any $n \geq 5$, $\lambda_{(3,2,1)}(C_n^2)\geq 12$; moreover, the following inequalities hold:
\begin{itemize}
    \item $\lambda_{(3,2,1)}(C_n^2)\leq 16$ if $7 \leq n \leq 40$;
    \item $\lambda_{(3,2,1)}(C_n^2)\leq 14$ if $n \geq 42$;
    \item $\lambda_{(3,2,1)}(C_n^2) \leq 15$ if $n \equiv 0 \pmod{19}$ or $n \equiv 0 \pmod{41}$;
    \item $\lambda_{(3,2,1)}(C_n^2)=12$ if $n \equiv 0 \pmod{7}$;
    \item $\lambda_{(3,2,1)}(C_n^2)\leq 14$ if $n \equiv 0 \pmod{8}$;
    \item $\lambda_{(3,2,1)}(C_n^2) \leq 13$ if $n \equiv 0 \pmod{10}$;
    \item $\lambda_{(3,2,1)}(C_n^2) \leq 15$ if $n \equiv 0 \pmod{11}$;
    \item $\lambda_{(3,2,1)}(C_n^2) \leq 16$ if $n\equiv 0 \pmod{13}$.
\end{itemize}

Finally, $\lambda_{(3,2,1)}(C_n^2)\leq 14$ when $n = 7m + q$ with $m \geq q \geq 1$.

\end{theorem}

\section{Circulants $C_n(1,3)$}

It makes sense to consider $n\geq 7$. We start with a lower bound on $\lambda_{(3,2,1)}$ of such circulants.

\begin{proposition} \label{C13 lower}
Let $n \in \mathbb{N}$, $n\geq 7$, and let $G=C_n(1,3)$. Then $\lambda_{(3,2,1)}(G)\geq 11$.
\end{proposition}

\begin{proof}
Consider the distance graph $G(1,3)$. We show that $\lambda_{(3,2,1)}(G(1,3))>10$. Suppose to the contrary that there is an $L(3,2,1)$-labeling $f'$ of $G(1,3)$ using labels from $\{0,1,\dots, 10\}$. Let $i$ denote any vertex of $G(1,3)$ having label $0$. Note that such a vertex $i$ must exist, since otherwise all labels could be decreased uniformly to obtain such a vertex. Let $G_1$ denote a subgraph of $G(1,3)$ induced by the set $\{i,i+1,i+2,\dots, i+7\}$ and $f$ be the restriction of $f'$ on $G_1$. Thus we want to label $8$ vertices of $G_1$ using $11$ labels from $\{0,1,\dots, 10\}$ such that $f(i)=0$. Clearly, $\mbox{diam}(G_1)=3$, thus all labels used on vertices of $G_1$ have to be distinct. If we do not use a triple of consecutive labels, then the only possibility is that $f$ uses labels $\{0,1,3,4,6,7,9,10\}$. 
In this set of labels, there are four disjoint pairs of consecutive integers. Each such pair must be assigned to vertices at distance at least $3$, because labels differing by $1$ require distance at least $3$ in an $L(3,2,1)$-labeling. However, in $G_1$ there are only three such pairs of vertices, a contradiction.
Thus $f$ has to use three consecutive labels $a,a+1,a+2$. Since $f(i)=0$, if $a\not=0$ then there are no three vertices at appropriate mutual distances among vertices $i+1,i+2,\dots, i+7$ to receive three consecutive labels greater than $0$. Thus $a=0$. Since $f(i)=0$, the only possibility is that $f(i+7)=1$ and $f(i+2)=2$ as any other assignment would violate the distance constraints of $L(3,2,1)$-labeling. Clearly, there is no vertex at distance $3$ from $i+2$ different from $i+7$, and therefore label $3$ cannot be used. Hence, for the remaining $5$ vertices we have only $7$ available labels $4,5,\dots, 10$. The subgraph of $G_1$ induced by $\{i+3,i+4,i+5,i+6\}$ is a $C_4$. From Theorem \ref{chromaticke_cislo_kruznic}, we have $\lambda_{(3,2,1)}(C_4)=7$, that is, it requires a span of $8$ labels, but only $7$ are available, a contradiction. Thus $\lambda_{(3,2,1)}(G_1)>10$.

For the sake of the proof, $G$ contains $G_1$ as a subgraph for any $n\geq 8$, hence $\lambda_{(3,2,1)}(G)\geq \lambda_{(3,2,1)}(G_1) \geq 11$ for every $n\geq 8$. In addition, for $n=7$ we have $\lambda_{(3,2,1)}(G)\geq 11$ from Table \ref{table 13 odd}.
\end{proof}

First we deal with $n$ even. In Table $\ref{table_1_3}$, we list exact values of $\lambda_{(3,2,1)}(C_n(1,3))$ for even $n<50$. For all graphs mentioned in Table \ref{table_1_3}, the exact values of $\lambda_{(3,2,1)}$ were obtained by a computer search. On one hand, the programme verified that the values of $\lambda_{(3,2,1)}$ cannot be reduced by showing that the corresponding graph admits no labeling using values $0,\dots, \lambda_{(3,2,1)}-1$. On the other hand, it gave a feasible labeling by a repetitive pattern, which is presented in Appendix of this paper.

\begin{table}[!ht]

\centering
\begin{tabular}{|c|c||c|c||c|c|}
  \hline
  $n$ & $\lambda_{(3,2,1)}(G)$ & $n$ & $\lambda_{(3,2,1)}(G)$ & $n$ & $\lambda_{(3,2,1)}(G)$\\
  \hline
  8 & 15 & 22 & 13 & 36 & 11\\
  10 & 13 & 24 & 11 & 38 & 13\\
  12 & 11 & 26 & 13 & 40 & 13\\
  14 & 13 & 28 & 13 & 42 & 13\\
  16 & 15 & 30 & 13 & 44 & 13\\
  18 & 13 & 32 & 13 & 46 & 13\\
  20 & 13 & 34 & 13 & 48 & 11\\
  \hline
\end{tabular}
\caption{The exact values of $\lambda_{(3,2,1)}(G)$ for circulant $G=C_n(1,3)$ with even $n<50$}
\label{table_1_3}
\end{table}

In the following statement we show that the lower bound from Proposition \ref{C13 lower} is the exact value for some specific values of $n$. 

\begin{theorem} \label{THM 13 11}
Let $k,n\in \mathbb{N}$, $n=12k$, and let $G=C_n(1,3)$. Then $\lambda_{(3,2,1)}(G)=11$.
\end{theorem}

\begin{proof}
The pattern $0,5,10,3,8,1,6,11,4,9,2,7$ forms an $L(3,2,1)$-labeling of $C_{12}(1,3)$ (see Fig. \ref{FIG C13}). For a labeling of $G$, copy this pattern $k$ times along the whole $G$. For each $u\in V(G)$, all neighbours of $u$ in $G$ have exactly the same labels as in $C_{12}(1,3)$. Analogously, vertices at distance $2$ from $u$ have the same labels in $G$ as in $C_{12}(1,3)$. For each $u_i,u_j\in V(G)$ with $\dist_G(u_i,u_j)=3$ we have $\vert i-j\vert \leq 9$ or $\vert i-j\vert \geq n-9$. Since the length of the pattern is $12$ and no label is repeated, $f(u_i)\not=f(u_j)$. Hence the defined labeling is an $L(3,2,1)$-labeling of $G$. Together with Proposition \ref{C13 lower} we obtain $\lambda_{(3,2,1)}(G)=11$.
\end{proof}

\begin{figure}[ht]
  \centering
  \includegraphics[width=0.4\textwidth]{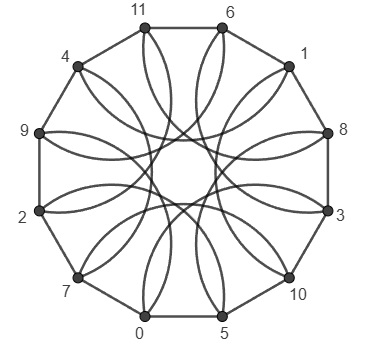}
  \caption{The graph $C_{12}(1,3)$ labeled using values from $0,1,\dots, 11$.}
  \label{FIG C13}
\end{figure}

For proving a general upper bound on $\lambda_{(3,2,1)}(C_n(1,3))$ we use two repetitive patterns presented in Proposition \ref{prop C13 upper}, one for $n=10$ and one for $n=14$. Using these two patterns we will generate an $L(3,2,1)$-labeling of $C_n(1,3)$ for any even $n\geq 48$.

\begin{proposition} \label{prop C13 upper}
Let $k,n\in\mathbb{N}$ and $G=C_n(1,3)$. If $n=10k$ or $n=14k$, then $\lambda_{(3,2,1)}(G)\leq 13$ and the following patterns form an $L(3,2,1)$-labeling of $G$
\begin{equation*}
P_1=(0,3,10,13,6,1,4,9,12,7), \qquad
P_2=(0,7,12,3,10,1,6,13,4,9),  \end{equation*} 

\vspace*{-8mm} 

\begin{equation*}
\mbox{and } \,\, P_3=(0,3,6,9,12,1,4,7,10,13) \quad \mbox{for } n=10k,
\end{equation*}
\begin{equation*} 
P_4=(0,3,6,9,12,1,4,7,10,13,2,5,8,11) \quad \mbox{for } n=14k.
\end{equation*}
\end{proposition}

\begin{proof}
Patterns $P_1$, $P_2$ and $P_3$ form an $L(3,2,1)$-labeling of $C_{10}(1,3)$, pattern $P_4$ forms an $L(3,2,1)$-labeling of $C_{14}(1,3)$ (all these patterns are illustrated in Fig. \ref{FIG C13 10} and \ref{FIG C13 14}). Since each of the patterns contains values from $\{0,1,\dots, 13\}$, we have $\lambda_{(3,2,1)}(G)\leq 13$ for $n=10$ and $n=14$.

Take a relevant pattern and copy it $k$ times along the whole $G$. We show that such a labeling is an $L(3,2,1)$-labeling of $G$. For each $u\in V(G)$, all neighbours of $u$ in $G$ have exactly the same labels as in the trivial circulant $C_{10}(1,3)$ or $C_{14}(1,3)$, respectively. Analogously, vertices at distance $2$ from $u$ have the same labels in $G$ as in the relevant trivial circulant. For each $u_i,u_j\in V(G)$ with $\dist_G(u_i,u_j)=3$ we have $\vert i-j\vert \leq 9$ or $\vert i-j\vert \geq n-9$. Since the length of any of the two patterns is at least 10 and no label is repeated in any of the patterns, $f(u_i)\not=f(u_j)$. Hence the defined labeling is an $L(3,2,1)$-labeling of $G$. 
\end{proof}

\begin{figure}[ht]
  \centering
  \includegraphics[width=1.0\textwidth]{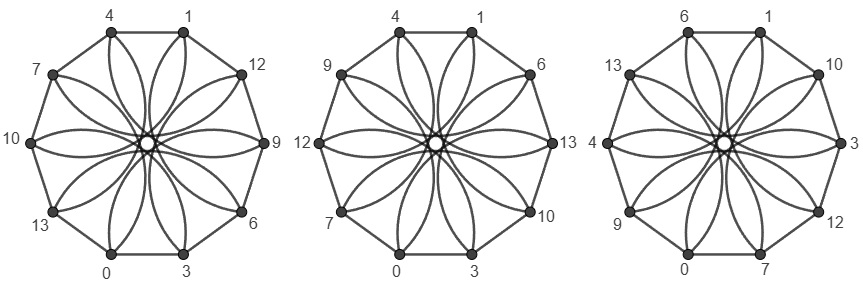}
  \caption{Circulant $C_{10}(1,3)$ with three different labelings using values $0,1,\dots, 11$}
  \label{FIG C13 10}
\end{figure}

\begin{figure}[ht]
  \centering
  \includegraphics[width=0.37\textwidth]{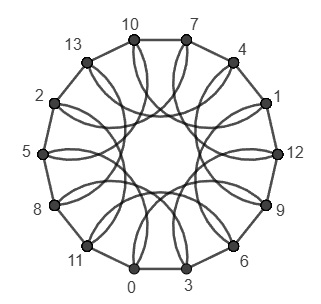}
  \caption{The graph $C_{14}(1,3)$ labeled using values from $0,1,\dots, 13$}
  \label{FIG C13 14}
\end{figure}

The following statement gives a general upper bound on $\lambda_{(3,2,1)}$ of $C_n(1,3)$ with even $n$.

\begin{theorem} \label{THM 13 even}
Let $n \in \mathbb{N}$, $n\geq 18$ be even, and let $G=C_n(1,3)$. Then $\lambda_{(3,2,1)}(G) \leq 13$.
\end{theorem}

\begin{proof}
For $18\leq n\leq 48$, the statement follows from Table \ref{table_1_3}. The relevant patterns are presented in Appendix. Hence we assume that $n\geq 50$. By the Sylvester's theorem (Theorem \ref{sylvester}), for every $n\geq 24$, there are nonnegative integers $k',\ell'$ such that $n=5k'+7\ell'$. Therefore, for every even integer $n\geq 48$, there are nonnegative integers $k,\ell$ such that $n=10k+14\ell$. We use two patterns presented in Proposition \ref{prop C13 upper} - pattern $P_1=(0,3,6,9,12,1,4,7,10,13)$ and pattern $P_2=(0,3,6,9,12,1,4,7,10,13,2,5,8,11)$. We define a labeling $f$ of the vertices of $G$ in such a way that we copy pattern $P_1$ $k$ times on vertices $u_1,u_2,\dots, \, u_{10k}$ and then we copy pattern $P_2$ $\ell$ times on vertices $u_{10k+1}, u_{10k+2}, \dots, \, u_n$. From Proposition \ref{prop C13 upper}, there is no collision in the defined labeling between any pair of vertices from $u_1,u_2,\dots, \, u_{10k}$ and any pair of vertices from $u_{10k+1}, u_{10k+2}, \dots, \, u_n$. It remains to show that the patterns for trivial circulants $C_{10}(1,3)$ and $C_{14}(1,3)$ can be combined into a feasible labeling of $G$. It suffices to show that there is no collision between two consecutive patterns of distinct lengths. We have the following sequences of labels (in the first line, the pattern of length $10$ is followed by the pattern of length $14$; in the second line, the patterns are switched). Note that the symbol $\vert$ means the border between the patterns.

\begin{equation*}
0,3,6,9,12,1,4,7,10,13, \,\, \vert \,\, 0,3,6,9,12,1,4,7,10,13,2,5,8,11,
\end{equation*}

\begin{equation*}
0,3,6,9,12,1,4,7,10,13,2,5,8,11, \,\, \vert \,\, 0,3,6,9,12,1,4,7,10,13.
\end{equation*} 

Obviously, the pattern of length $10$ is a subsequence of the pattern of length $14$, therefore we infer that there is no collision in the labeling $f$, and hence $f$ is an $L(3,2,1)$-labeling of $G$.
\end{proof}

For large $n$, the upper bound stated in Theorem \ref{THM 13 even} can be improved. For proving that, we will use the following observation.

\begin{observ}
\label{vzory_1_3}
Let $G=C_n(1,3)$, where $n \in \{50, 76, 126, 130, 152\}$. Then $\lambda_{(3,2,1)}(G)\leq 12$ and the following sequences of labels form relevant $L(3,2,1)$-labelings:

\begin{equation*}
\begin{array}{ll}
n & \mbox{pattern} \\
50 & 0, 5, 10, 3, 8, 1, 6, 11, 4, 9, 2, 7, 12, 5, 0, 3, 8, 11, 6, 1, 4, 9, 12, 7, 2, 5, 0, 11, 8, \\ & 3, 6, 1, 12, 9, 4, 7, 2, 11, 0, 5, 8, 3, 12, 1, 6, 9, 4, 11, 2, 7; \\

76 & 0, 5, 10, 3, 8, 1, 6, 11, 4, 9, 2, 7, 12, 5, 10, 1, 8, 3, 6, 11, 0, 9, 4, 7, 12, 1, 10, 5, 8, \\ & 3, 0, 11, 6, 9, 4, 1, 12, 7, 10, 5, 0, 3, 8, 11, 6, 1, 4, 9, 12, 7, 2, 5, 0, 11, 8, 3, 6, 1,  \\ & 12, 9, 4, 7, 2, 11, 0, 5, 8, 3, 12, 1, 6, 9,
4, 11, 2, 7; \\

126 & 0, 5, 10, 3, 8, 1, 6, 11, 4, 9, 2, 7, 12, 5, 0, 3, 8, 11, 6, 1, 4, 9, 12, 7, 2, 5, 0, 11, 8, 3,  \\ & 6, 1, 12, 9, 4, 7, 2, 11, 0, 5, 8, 3, 12, 1, 6, 9, 4, 11, 2, 7, 0, 5, 10, 3, 8, 1, 6, 11, 4, 9, \\ & 0, 7, 12, 5, 10, 1, 8, 3, 6, 11, 0, 9, 4, 7, 12, 1, 10, 5, 8, 3, 0, 11, 6, 9, 4, 1, 12, 7, 10,  \\ & 5, 0, 3, 8, 11, 6, 1, 4, 9, 12, 7, 2, 5, 0, 11, 8, 3, 6, 1, 12, 9, 4, 7, 2, 11, 0, 5, 8, 3, 12, \\ & 1, 6, 9, 4, 11, 2, 7; \\

130 & 0, 5, 10, 3, 8, 1, 6, 11, 4, 9, 2, 7, 12, 5, 10, 3, 8, 0, 6, 11, 4, 9, 1, 7, 12, 5, 10, 2, 8,  \\ & 0, 6, 11,  3, 9, 1, 7, 12, 4, 10, 2, 8, 0, 5, 11, 3, 9, 1, 6, 12, 4, 10, 2, 7, 0, 5, 11, 3, 8,   \\ & 1, 6, 12, 4, 9, 2, 7, 0, 5, 10, 3, 8, 1, 6, 11, 4, 9, 2, 7, 12, 5, 10, 3, 8, 0, 6, 11, 4, 9, \\ & 1, 7, 12, 5, 10, 2, 8, 0, 6, 11, 3, 9, 1, 7, 12, 4, 10, 2, 8, 0, 5, 11, 3, 9, 1, 6, 12, 4, 10, \\ & 2, 7, 0,  5, 11, 3, 8, 1, 6, 12, 4,  9, 2, 7; 
\end{array}
\end{equation*}

\vspace{2cm} \mbox{}

\begin{equation*}
\begin{array}{ll}
n & \mbox{pattern} \\
152 & 0, 5, 10, 3, 8, 1, 6, 11, 4, 9, 2, 7, 12, 5, 10, 1, 8, 3, 6, 11, 0, 9, 4, 7, 12, 1, 10, 5, 8, \\ &  3, 0, 11,  6, 9, 4, 1, 12, 7, 10, 5, 0, 3, 8, 11, 6, 1, 4, 9, 12, 7, 2, 5, 0, 11, 8, 3, 6, 1,  \\ & 12, 9, 4, 7, 2, 11, 0, 5, 8, 3, 12, 1, 6, 9, 4, 11, 2, 7, 0, 5, 10, 3, 8, 1, 6, 11, 4, 9, 2,  \\ & 7, 12, 5, 10, 1, 8, 3, 6, 11, 0, 9, 4, 7, 12, 1, 10, 5, 8, 3, 0, 11, 6, 9, 4, 1, 12, 7, 10,  \\ & 5, 0, 3, 8, 11,  6, 1, 4, 9, 12, 7, 2, 5, 0, 11, 8, 3, 6, 1, 12, 9, 4, 7, 2, 11, 0, 5, 8, 3,  \\ &  12, 1, 6, 9, 4, 11, 2, 7.
\end{array}
\end{equation*}
\end{observ}

\begin{proof}
All the patterns were obtained and verified by a computer search. Clearly, each of the patterns uses labels $0,1,\dots, 12$, hence $\lambda_{(3,2,1)}(G)\leq 12$. 

\end{proof}

\begin{theorem} \label{THM 13 even better}
Let $n\in \mathbb{N}$, $n\geq 142$ be even, and let $G=C_n(1,3)$. Then $\lambda_{(3,2,1)}(G)\leq 12$.
\end{theorem}

\begin{proof}
In the proof, we will consider possibilities depending on the remainder $r$ after division of $n$ by $12$. Let $n=12k+g(r)$ for some $k\in\mathbb{N}$, where $g(0)=0$, $g(2)=50$, $g(4)=76$,  $g(6)=126$, $g(8)=152$ and $g(10)=130$. For $g(0)=0$ we label the  vertices of $G$ according to the proof of Theorem $\ref{THM 13 11}$ and we have $\lambda_{(3,2,1)}(G)\leq 12$.

In the remaining cases, we define a labeling $f$ of the vertices of $G$ in such a way that we copy the pattern $0,5,10,3,8,1,6,11,4,9,2,7$ given in the proof of Theorem $\ref{THM 13 11}$ $k$ times on vertices $u_1,u_2,\dots, u_{12k}$, and then we use the relevant pattern given in Observation $\ref{vzory_1_3}$ for corresponding $g(r)$. It follows that the only possible collision in the defined labeling of $G$ could be between some vertex labeled by the pattern of length $12$ and some vertex labeled by the pattern of length $g(r)$. Specifically, it suffices to show that there is no collision between any vertex from $\{u_1,u_2,\dots, \, u_9\}$ and from $\{ u_{n-8}, u_{n-7}, \dots, \, u_{n}\}$, and between any vertex from $\{ u_{10k-8}, u_{10k-7}, \dots, \, u_{10k}\}$ and from $\{u_{10k+1}, u_{10k+2},\dots, \, u_{10k+9}\}$, since $\dist_G(u_i,u_j)>3$ for every $i,j$ with $\vert i-j\vert>9$ and $\vert i-j\vert <n-9$.

Obviously, the pattern of length $12$ is a subsequence of each of the patterns given in Observation \ref{vzory_1_3} (first $12$ values), therefore we infer that there is no collision in the labeling $f$, and thus $f$ is an $L(3,2,1)$-labeling of $G$.
Since $12 \equiv 0 \,(\mbox{mod } 12)$, $50\equiv 2 \,(\mbox{mod } 12)$, $76\equiv 4\, (\mbox{mod } 12)$, $126 \equiv 6 \,(\mbox{mod } 12)$, $152 \equiv 8 \,(\mbox{mod } 12)$ and $130 \equiv 10\, (\mbox{mod } 12)$, we discussed all even $n$. Therefore $\lambda_{(3,2,1)}(G)\leq 12$ for every even $n \geq 142$.
\end{proof}

Note that, for $n \in \{38, 64, 114, 118, 140\}$, we have $\lambda_{(3,2,1)}(G)=13$ by a computer search, thus the assumption on $n$ in the previous theorem cannot be decreased. We also believe that the following conjecture holds true. We tested all even integers $n\leq 500$ and we found that $\lambda_{(3,2,1)}(G)=11$ if and only if $n$ is divisible by $12$.

\begin{conjecture}
Let $n\in \mathbb{N}$, $n\geq 142$ be even, and let $G=C_n(1,3)$. If $n$ is not divisible by $12$, then $\lambda_{(3,2,1)}(G) =12$.
\end{conjecture}

Now we move on to odd $n$. In Table $\ref{table 13 odd}$, a list of exact values of $\lambda_{(3,2,1)}(C_n(1,3))$ for odd $n<50$ is given. These values were obtained by a computer search. For all graphs mentioned in Table \ref{table 13 odd}, the exact values of $\lambda_{(3,2,1)}$ were obtained by a computer search. On one hand, the programme verified that the values of $\lambda_{(3,2,1)}$ cannot be reduced by showing that the corresponding graph admits no labeling using values $0,\dots, \lambda_{(3,2,1)}-1$. On the other hand, it gave a feasible labeling by a repetitive pattern, which is presented in Appendix of this paper.

\begin{table}[!ht]

\centering
\begin{tabular}{|c|c||c|c||c|c|}

  \hline
  $n$ & $\lambda_{(3,2,1)}(G)$ & $n$ & $\lambda_{(3,2,1)}(G)$ & $n$ & $\lambda_{(3,2,1)}(G)$\\
  \hline 
  7  & 12 & 23 & 15 & 39 & 14\\
  9  & 16 & 25 & 14 & 41 & 13\\
  11 & 20 & 27 & 14 & 43 & 14\\
  13 & 18 & 29 & 14 & 45 & 14\\
  15 & 14 & 31 & 15 & 47 & 14\\
  17 & 16 & 33 & 15 & 49 & 14\\
  19 & 18 & 35 & 14 &  & \\
  21 & 16 & 37 & 14 &  &  \\
  \hline
\end{tabular}
\caption{The exact values of $\lambda_{(3,2,1)}(G)$ for circulant $G=C_n(1,3)$ with odd $n<50$}
\label{table 13 odd}
\end{table}

For proving a general upper bound on $\lambda_{(3,2,1)}$ of $C_n(1,3)$ when $n$ is odd, we will use the following observation and repetitive patterns for $n=10$ presented in Proposition \ref{prop C13 upper}.

\begin{observ} \label{obs C13 vzory odd}
Let $G=C_n(1,3)$, where $n\in \{51,53,55,57,59,67,69\}$. Then $\lambda_{(3,2,1)}(G)\leq 13$ and the following sequences of labels form relevant $L(3,2,1)$-labelings:
\begin{equation*}
\begin{array}{ll}
n & \mbox{pattern} \\
51 & 0,3,10,13,6,1,4,9,12,7,0,5,10,3,8,13,6,11,1,9,4,7,12,0,10,2,8,13,\\ & 5,11,3,9,1,6,12,4,10,2,7,13,5,0,3,8,11,6,1,4,9,12,7; \\
53 & 0,7,12,5,10,1,8,13,6,11,3,9,0,7,12,4,10,2,8,13,5,11,3,9,1,6,12,4,\\ & 10,2,7,13,5,0,3,8,11,6,1,4,9,12,7,0,3,10,5,8,13,2,11,4,9; \\
55 & 0,3,8,11,6,1,4,9,12,7,0,5,10,3,8,13,6,11,1,9,4,7,12,0,10,2,8,13,5,\\ & 11,3,9,1,6,12,4,10,2,7,13,5,0,3,8,11,6,1,4,9,12,7,2,5,10,13;\\
57 & 0,3,8,5,10,1,12,7,4,9,0,11,6,13,8,2,10,5,12,7,1,9,3,11,6,0,8,2,10,\\ & 4,13,7,1,9,3,12,5,0,8,2,11,4,13,6,1,10,3,12,5,0,7,2,9,4,11,6,13; \\\end{array}
\end{equation*}

\begin{equation*}
\begin{array}{ll}
n & \mbox{pattern} \\

59 & 0,3,12,5,8,1,10,13,4,7,0,9,2,11,6,13,8,1,10,4,12,7,0,9,3,11,5,13,\\ & 8,2,10,4,12,6,1,9,3,11,5,0,7,2,10,4,13,6,1,8,3,12,5,0,7,2,11,4,13,\\ & 6,9; \\
67 & 0,3,6,9,12,1,4,7,10,13,0,3,8,11,6,1,4,9,12,7,2,5,10,13,8,0,6,11,3,\\ & 9,1,7,12,4,10,2,8,0,5,11,3,13,1,6,9,4,12,2,7,0,5,10,3,8,13,6,11,0,\\ & 9,4,7,12,1,10,5,8,13; \\
69 & 0,3,6,9,12,1,4,7,10,13,0,3,8,11,6,1,4,9,12,7,2,5,10,13,8,0,6,11,3,\\ & 9,1,7,12,4,10,2,8,0,5,11,3,13,1,6,9,4,12,2,7,0,5,10,3,8,13,6,11,0,\\ & 9,2,5,12,7,10,1,4,13,8,11.\\
\end{array}
\end{equation*}
\end{observ}

\begin{proof}
All the patterns were obtained and verified by a computer search. Obviously, in each of the pattern, the maximum value is $13$ while the minimum value is $0$, hence $\lambda_{(3,2,1)}(G)\leq 13$ for every $n$ listed in the assumptions.
\end{proof}

\begin{theorem} \label{THM 13 odd}
Let $n\in \mathbb{N}$, $n\geq 51$ be odd, and let $G=C_n(1,3)$. Then $ \lambda_{(3,2,1)}(G) \leq 13$.
\end{theorem}

\begin{proof}
In the proof, we will consider possibilities depending on the remainder $r$ after division of $n$ by $10$. Let $n=10k+g(r)$ for some $k\in\mathbb{N}$, where $g(1)=51$, $g(3)=53$, $g(5)=55$, $g(7)=67$ and $g(9)=69$. (Note that, for $n=57$ and $n=59$, the statement is true by Observation \ref{obs C13 vzory odd}.) In each of the cases, we label vertices of $G$ in such a way that we copy one of the patterns for $C_{10}(1,3)$ given in Proposition \ref{prop C13 upper} $k$ times on vertices $u_1,u_2,\dots, u_{10k}$ and then use the pattern given in Observation \ref{obs C13 vzory odd} for relevant $g(r)$. From Proposition \ref{prop C13 upper}, it follows that the only possible collision in the defined labeling of $G$ could be between some vertex labeled by the pattern of length $10$ and some vertex labeled by the pattern of length $g(r)$. Specifically, it suffices to show that there is no collision between any vertex from $\{u_1,u_2,\dots, \, u_9\}$ and from $\{ u_{n-8}, u_{n-7}, \dots, \, u_{n}\}$ and between any vertex from $\{ u_{10k-8}, u_{10k-7}, \dots, \, u_{10k}\}$ and from $\{u_{10k+1}, u_{10k+2},\dots, \, u_{10k+9}\}$, since $\dist_G(u_i,u_j)>3$ for every $i,j$ with $\vert i-j\vert>9$ and $\vert i-j\vert <n-9$.
Now we distinguish the following cases:

\begin{enumerate}

\item[{\textbf Case 1:}\hspace*{-0.7cm}] \hspace{0.7cm}{ $r=1$}.\\
We use the pattern $0,3,10,13,6,1,4,9,12,7,$ as the pattern of length $10$. We have the following sequences of labels (in the first line, there is the pattern of length $10$ followed by the pattern for $n=51$; in the second line, the patterns are switched). Note that the symbol $|$ means the border between the patterns.
\begin{equation*}
\dots,3,10,13,6,1,4,9,12,7, \,\,\vert \,\,  0,3,10,13,6,1,4,9,12, \dots 
\end{equation*}
\begin{equation*}
\dots,3,8,11,6,1,4,9,12,7, \,\, \vert \,\, 0,3,10,13,6,1,4,9,12, \dots  
\end{equation*}

\item[{\textbf Case 2:}\hspace*{-0.7cm}] \hspace{0.7cm}{ $r=3$}.\\
We use the pattern $0,7,12,3,10,1,6,13,4,9,$ as the pattern of length $10$. We have the following sequences of labels (in the first line, there is the pattern of length $10$ followed by the pattern for $n=53$; in the second line, the patterns are switched). Note that the symbol $|$ means the border between the patterns.
\begin{equation*}
\dots, 7,12,3,10,1,6,13,4,9, \,\,\vert \,\, 0,7,12,5,10,1,8,13,6, \dots 
\end{equation*}
\begin{equation*}
\dots,3,10,5,8,13,2,11,4,9, \,\, \vert \,\, 0,7,12,3,10,1,6,13,4, \dots  
\end{equation*}

\item[{\textbf Case 3:}\hspace*{-0.7cm}] \hspace{0.7cm} {$r=5$}.\\
We use the pattern $0,3,6,9,12,1,4,7,10,13,$ as the pattern of length $10$. We have the following sequences of labels (in the first line, there is the pattern of length $10$ followed by the pattern for $n=55$; in the second line, the patterns are switched). 
\begin{equation*}
\dots, 3,6,9,12,1,4,7,10,13, \,\,\vert \,\, 0,3,8,11,6,1,4,9,12, \dots 
\end{equation*}
\begin{equation*}
\dots, 1,4,9,12,7,2,5,10,13, \,\, \vert \,\, 0,3,6,9,12,1,4,7,10, \dots  
\end{equation*}

\item[{\textbf Case 4:}\hspace*{-0.7cm}] \hspace{0.7cm} {$r=7$}.\\
Now we use the pattern $0,3,6,9,12,1,4,7,10,13,$ as the pattern of length $10$. We have the following sequences of labels (in the first line, there is the pattern of length $10$ followed by the pattern for $n=67$; in the second line, the patterns are switched).
\begin{equation*}
\dots,3,6,9,12,1,4,7,10,13, \,\,| \,\, 0,3,6,9,12,1,4,7,10, \dots  
\end{equation*}
\begin{equation*}
\dots,9,4,7,12,1,10,5,8,13, \,\, | \,\, 0,3,6,9,12,1,4,7,10, \dots  
\end{equation*}

\item[{\textbf Case 5:}\hspace*{-0.7cm}] \hspace{0.7cm} {$r=9$}.\\
In this case we use the pattern $0,3,6,9,12,1,4,7,10,13$ as the pattern of length $10$. We have the following sequences of labels (in the first line, there is the pattern of length $10$ followed by the pattern for $n=69$; in the second line, the patterns are switched).
\begin{equation*}
\dots,3,6,9,12,1,4,7,10,13, \,\,| \,\, 0,3,6,9,12,1,4,7,10 \dots  
\end{equation*}
\begin{equation*}
\dots, 5,12,7,10,1,4,13,8,11, \,\, | \,\, 0,3,6,9,12,1,4,7,10, \dots  
\end{equation*}

\end{enumerate}

In each of the cases, two vertices of $G$ are adjacent if and only if two numbers in any of the above listed sequences are next to each other or they have exactly two number in-between. Also, two vertices of $G$ are at distance $2$ if and only if two numbers in any of the above listed sequences have exactly one, three or five numbers in-between. And, two vertices of $G$ are at distance $3$ if and only if the numbers in the listed sequences have exactly four, six or eight numbers in-between. Clearly, in every sequence, labels of each pair of adjacent vertices differ by at least $3$, and labels of each pair of vertices at distance $2$ apart differ by at least $2$. For each $u,v\in V(G)$ with $\dist_G(u,v)=3$, labels of $u$ and $v$ are distinct since any pair of the same value in each of the sequences has at least $9$ other values in-between. Therefore the defined labeling is an $L(3,2,1)$-labeling of $G$ in each of the cases.
\end{proof}

We finish this section by summarizing all results presented in Proposition \ref{C13 lower} and Theorems \ref{THM 13 11}, \ref{THM 13 even}, \ref{THM 13 even better}
 and \ref{THM 13 odd}. Note that the presented bounds and exact values are supplemented by Tables \ref{table_1_3} and \ref{table 13 odd} for small values of $n$.
 
\begin{corollary}
Let $n\in \mathbb{N}$, 
and let $G=C_n(1,3)$. Then
$$
\begin{array}{rl}\lambda_{(3,2,1)}(G) =11 & \mbox{when } n \mbox{ is divisible by } 12, \\
11\leq \lambda_{(3,2,1)}(G) \leq 13 & \mbox{when } n\geq 18 \mbox{ is even}, \\
11\leq \lambda_{(3,2,1)}(G) \leq 12 & \mbox{when } n\geq 142 \mbox{ is even}, \\
11\leq \lambda_{(3,2,1)}(G) \leq 13 & \mbox{when } n\geq 51 \mbox{ is odd}.  \\
\end{array}
$$
\end{corollary}

\section{Circulants $C_n(1,4)$}
For such circulants, it make sense to consider $n\geq 9$. First we give a lower bound on $\lambda_{(3,2,1)}(G)$.

\begin{proposition}
\label{PROP C14 lower}
Let $n \in \mathbb{N}, n \geq 9$, and $G=C_n(1,4)$. Then $\lambda_{(3,2,1)}(G) \geq 15$.
\end{proposition}

\begin{proof}
For the distance graph $G(1,4)$, using a computer check we found that $\lambda_{(3,2,1)}(G(1,4)) \geq 15$. We tried to label vertices of $G(1,4)$ with labels $0,1,\dots, 14$, but the programme was able to label only $13$ consecutive vertices of $G(1,4)$ using these labels. The programme began by assigning label $0$ to a fixed vertex $i$ of $G(1,4)$ and then attempted to extend this labeling to vertices $i+1,i+2,i+3,\dots$ Note that such a vertex $i$ must exist, since otherwise all labels could be decreased uniformly, reducing the span. 

Let $G_1$ denote a subgraph of $G(1,4)$ induced by vertices $1,2,\dots, 14$. We know that $G_1$ has no $L(3,2,1)$-labeling using labels $0,1,\dots, 14$ such that vertex $1$ has label $0$. Since every $C_n(1,4)$ ($n>13$) contains $G_1$ as a subgraph, we have $\lambda_{(3,2,1)}(G)\geq 15$ for every $n>13$. From Table \ref{table_1_4}, the lower bound $15$ is valid also for $n \in \{9,10,11,12,13\}$.
\end{proof}

For some specific values of $n$, the lower bound from Proposition $\ref{PROP C14 lower}$ is the exact value of $\lambda_{(3,2,1)}$.

\begin{theorem}
\label{THM C14 presne}
Let $k,n \in \mathbb{N}$ and $G=C_n(1,4)$. If $n=16k$, then $\lambda_{(3,2,1)} (G) = 15$.
\end{theorem}

\begin{proof}
The pattern $0,5,10,15,3,8,13,1,6,11,4,9,14,2,7,12$ forms an $L(3,2,1)$-labeling of $C_{16}(1,4)$ (see Fig. \ref{FIG C14}). Hence, for $n=16$, we have $\lambda_{(3,2,1)}(G)=15$.

Take this pattern and copy it $k$ times along the whole $G$. We show that such a labeling is an $L(3,2,1)$-labeling of $G$. For each vertex $u$ of $G$, all neighbours of $u$ in $G$ have exactly the same labels as in $C_{16}(1,4)$. Analogously, vertices at distance $2$ from $u$ have the same labels in $G$ as in $C_{16}(1,4)$. For every $u_i,u_j\in V(G)$ with $\dist_G(u_i,u_j)=3$ we have $|j-i|\leq 12$ or $|j-i|\geq n-12$. Since the length of the repetitive pattern is $16$ and no label is repeated in this pattern, $f(u_i)\neq f(u_j)$. Thus the defined labeling is an $L(3,2,1)$-labeling of $G$. Together with Proposition $\ref{PROP C14 lower}$, we have $\lambda_{(3,2,1)}(G)=15$.
\end{proof}

\begin{figure}[ht]
  \centering
  \includegraphics[width=0.45\textwidth]{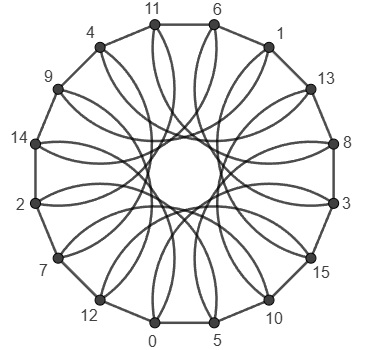}
  \caption{The graph $C_{16}(1,4)$ labeled using values from $0,1,\dots,15$.}
  \label{FIG C14}
\end{figure}

We continue with a table of exact values of $\lambda_{(3,2,1)}(C_n(1,4))$ for small values of $n$ (see Table $\ref{table_1_4}$). These exact values were obtained by computer experiments. We determined that the span of feasible labelings of such graphs cannot be reduced, and the reported values were confirmed through the identification and verification of relevant patterns (see Appendix).

For proving a general upper bound on $\lambda_{(3,2,1)}$ of $C_n(1,4)$, we will use the following two auxiliary statements.

\begin{table}
\centering
\begin{tabular}{|c|c||c|c||c|c|}

  \hline
  $n$ & $\lambda_{(3,2,1)}(G)$ & $n$ & $\lambda_{(3,2,1)}(G)$ & $n$ & $\lambda_{(3,2,1)}(G)$\\
  \hline \hline
  9 & 16 & 25 & 18 & 41 & 17\\
  10 & 18 & 26 & 17 & 42 & 16\\
  11 & 20 & 27 & 17 & 43 & 17\\
  12 & 16 & 28 & 17 & 44 & 17\\
  13 & 18 & 29 & 17 & 45 & 16\\
  14 & 19 & 30 & 16 & 46 & 16\\
  15 & 16 & 31 & 16 & 47 & 16\\
  16 & 15 & 32 & 15 & 48 & 15\\
  17 & 16 & 33 & 16 & 49 & 16\\
  18 & 17 & 34 & 16 & 50 & 16\\
  19 & 18 & 35 & 17 & 51 & 16\\
  20 & 18 & 36 & 17 & 52 & 17\\
  21 & 20 & 37 & 18 & 53 & 17\\
  22 & 20 & 38 & 17 & 54 & 17\\
  23 & 18 & 39 & 17 & 55 & 17\\
  24 & 18 & 40 & 17 & 56 & 17\\
  \hline
\end{tabular}
\caption{The exact values of $\lambda_{(3,2,1)}(G)$ for circulant $G=C_n(1,4)$ with $n\leq56$}
\label{table_1_4}
\end{table}

\begin{observ}
\label{C14 dlouhe vzory vyjimka}
Let $G=C_{n}(1,4)$, where $n\in \{60,69,70,71\}$. Then $\lambda_{(3,2,1)}(G)\leq 16$ and the following sequences of labels form relevant $L(3,2,1)$-labelings:

\begin{equation*}
\begin{array}{ll}
n & \mbox{pattern} \\
60 &  0, 6, 12, 2, 8, 14, 4, 10, 16, 1, 7, 13, 3, 9, 15, 0, 6, 12, 2, 8, 14, 4, 10, 16, 1, 7, 13, 3, \\ & 9, 15, 0, 6, 12, 2, 8, 14, 4, 10, 16, 1, 7, 13, 3, 9, 15, 0, 6, 12, 2, 8, 14, 4, 10, 16, 1, \\ & 7, 13, 3, 9, 15;\\

69 &  0, 5, 11, 16, 3, 8, 14, 1, 6, 12, 4, 9, 15, 0, 7, 13, 2, 10, 16, 5, 8, 14, 3, 11, 0, 6, 9, 15, \\ & 4, 12, 1, 7, 10, 16, 5, 13, 2, 8, 11, 0, 6, 14, 3, 9, 16, 1, 7, 12, 4, 10, 15, 2, 8, 13, 0, \\ & 5, 11, 16, 3, 14, 1, 6, 12,  8, 4, 10, 15, 2, 13;\\

70 &  0, 3, 8, 13, 16, 6, 11, 1, 9, 14, 4, 7, 12, 2, 16, 10, 0, 5, 8, 13, 3, 15, 11, 1, 6, 9, 4, \\ & 16, 12, 2, 14, 7, 0, 5, 10, 3, 15, 8, 1, 13, 6, 11, 16, 9, 2, 14, 4, 0, 12, 7, 10, 15, 5, 1, \\ & 13, 3, 8, 11, 16, 6, 0, 14, 4, 9, 12, 2, 7, 15, 5, 10;\\

71 & 0, 3, 10, 15, 5, 8, 13, 1, 11, 16, 4, 9, 14, 2, 7, 12, 0, 5, 10, 16, 3, 8, 14, 1, 6, 12, 4,\\ & 9, 15, 0, 7, 13, 2, 10, 16, 5, 8, 14, 3, 0, 12, 6, 9, 15, 4, 1, 13, 7, 10, 16, 5, 2, 14, 8, \\ & 0, 12, 6, 3, 10, 15, 1, 13, 7, 4, 11, 16, 2, 9, 14, 6, 12.
\end{array}
\end{equation*}

\end{observ}

\begin{proof}
All the patterns were obtained and verified by a computer search. Obviously, in each of the pattern, the maximum value is $16$ while the minimum value is $0$, hence $\lambda_{(3,2,1)}(G)\leq 16$ for every $n$ listed in the assumptions.
\end{proof}

\begin{observ}
\label{C14 dlouhe vzory}
Let $G=C_{n}(1,4)$, where $n\in \{17,31,34,42,46,51,57,59,61,\\ 68,72,76,85,86,87\}$. Then $\lambda_{(3,2,1)}(G)\leq 16$ and the following sequences of labels form relevant $L(3,2,1)$-labelings:

\begin{equation*}
\begin{array}{ll}
n & \mbox{pattern} \\
17 & 0, 5, 10, 15, 3, 8, 13, 1, 6, 11, 16, 4, 9, 14, 2, 7, 12;\\

31 & 0, 5, 10, 15, 3, 8, 13, 1, 6, 11, 4, 9, 14, 2, 7, 12, 0, 5, 10, 3, 8, 13, 1, 6, 11, 16, 4, 9, \\ & 2, 7, 12; \\

34 & 0, 5, 10, 15, 3, 8, 13, 1, 6, 11, 16, 4, 9, 14, 2, 7, 12, 0, 5, 10, 15, 3, 8, 13, 1, 6, 11, \\ & 16, 4, 9, 14, 2, 7, 12;\\

42 & 0, 5, 10, 3, 16, 8, 1, 14, 6, 12, 4, 9, 2, 15, 7, 0, 13, 5, 11, 3, 16, 9, 1, 14, 7, 12, 4, 10,\\ &  2, 15, 8, 0, 13, 6, 11, 16, 4, 9, 14, 2, 7, 12;\\

46 & 0, 5, 10, 15, 3, 8, 13, 1, 6, 11, 4, 9, 14, 2, 7, 12, 0, 5, 10, 3, 8, 13, 1, 6, 11, 16, 4, 9, \\ & 2, 7, 12, 0, 5, 10, 15, 3, 8, 1, 6, 11, 14, 4, 9, 16, 2, 7;\\

51 & 0, 5, 10, 15, 3, 8, 13, 1, 6, 11, 16, 4, 9, 14, 2, 7, 12, 0, 5, 10, 15, 3, 8, 13, 1, 6, 11,\\ &  16, 4, 9, 14, 2, 7, 12, 0, 5, 10, 15, 3, 8, 13, 1, 6, 11, 16, 4, 9, 14, 2, 7, 12; \\
%
57 & 0, 5, 10, 15, 3, 8, 13, 1, 6, 11, 4, 9, 14, 2, 7, 12, 0, 5, 10, 3, 16, 8, 1, 14, 6, 12, 4, 9, 2, \\ & 15, 7, 0, 13, 5, 11, 3, 16, 9, 1, 14, 7, 12, 4, 10, 2, 15, 8, 0, 13, 6, 11, 16, 4, 9, 14, 2, 7;\\

59 & 0, 5, 10, 3, 16, 8, 1, 14, 6, 12, 4, 9, 2, 15, 7, 0, 13, 5, 11, 3, 16, 9, 1, 14, 7, 12, 4, 10, \\ &  2, 15, 8, 0, 13, 6, 11, 16, 4, 9, 14, 2, 7, 12, 0, 5, 10, 15, 3, 8, 13, 1, 6, 11, 16, 4, 9, \\ & 14, 2, 7, 12;\\

61 & 0, 5, 10, 15, 3, 8, 13, 1, 6, 11, 4, 9, 14, 2, 7, 12, 0, 5, 10, 3, 8, 13, 1, 6, 11, 16, 4, 9, \\ & 2, 7, 12, 0, 5, 10, 15, 3, 8, 1, 6, 11, 16, 4, 9, 14, 2, 12, 0, 5, 10, 15, 3, 8, 13, 6, 11, \\ & 16, 4, 9, 14, 2, 7;\\

68 & 0, 5, 10, 15, 3, 8, 13, 1, 6, 11, 16, 4, 9, 14, 2, 7, 12, 0, 5, 10, 15, 3, 8, 13, 1, 6, 11, 16,\\ & 4, 9, 14, 2, 
 7, 12, 0, 5, 10, 15, 3, 8, 13, 1, 6, 11, 16, 4, 9, 14, 2, 7, 12, 0, 5, 10, 15, \\ & 3, 8, 13, 1, 6, 11, 16, 4, 9, 14, 2, 7, 12,\\

72 & 0, 5, 10, 15, 3, 8, 13, 1, 6, 11, 4, 9, 14, 2, 7, 12, 0, 5, 15, 3, 8, 13, 1, 6, 11, 16, 9, 14, \\ & 2, 7, 12, 0, 5, 10,
 3, 16, 8, 1, 14, 6, 12, 4, 9, 2, 15, 7, 0, 13, 5, 11, 3, 16, 9, 1, 14, 7, \\ & 12, 4, 10, 2, 15, 8, 0, 13,  6, 11, 16, 4, 9, 14, 2, 7;\\

76 & 0, 5, 10, 15, 3, 8, 13, 1, 6, 11, 4, 9, 14, 2, 7, 12, 0, 5, 10, 3, 8, 13, 1, 6, 11, 16, 4, 9, 2, \\ & 7, 12, 0, 5, 10,
 15, 3, 13, 1, 6, 11, 16, 4, 9, 14, 7, 12, 0, 5, 10, 15, 3, 8, 1, 6, 11, 16, \\ & 4, 9, 14, 2, 12, 0, 5, 10, 15, 3, 8, 13, 6, 11, 16, 4, 9, 14, 2, 7;\\

85 & 0, 5, 10, 15, 3, 8, 13, 1, 6, 11, 16, 4, 9, 14, 2, 7, 12, 0, 5, 10, 15, 3, 8, 13, 1, 6, 11, 16, \\ & 4, 9, 14, 2,
 7, 12, 0, 5, 10, 15, 3, 8, 13, 1, 6, 11, 16, 4, 9, 14, 2, 7, 12, 0, 5, 10, 15, \\ & 3, 8, 13, 1, 6, 11, 16, 4, 9, 14, 2, 7, 12, 0, 5, 10, 15, 3, 8, 13, 1, 6, 11, 16, 4, 9, 14, 2, \\ &  7, 12;\\

86 & 0, 5, 10, 16, 3, 8, 14, 1, 6, 12, 4, 9, 15, 0, 7, 13, 2, 10, 16, 5, 8, 14, 3, 0, 12, 6, 9, 15, \\ & 4, 1, 13, 7,
 10, 16, 5, 2, 14, 8, 0, 12, 6, 3, 15, 9, 1, 13, 7, 4, 11, 16, 2, 14, 8, 5, 12, \\ & 0, 3, 10, 15, 7, 13, 1, 4, 11, 16, 6, 9, 14, 2, 12, 0, 5, 10, 15, 3, 8, 13, 1, 6, 11, 4, 9, 14, \\ & 2, 7, 12;\\
 
%

87 & 0, 5, 10, 15, 8, 13, 1, 4, 11, 16, 6, 9, 14, 3, 12, 0, 5, 10, 15, 7, 2, 13, 4, 11, 16, 8, 0, \\ &  14, 6, 3, 12,  
 9, 1, 15, 7, 4, 13, 10, 2, 16, 8, 0, 5, 12, 3, 14, 9, 1, 6, 11, 16, 4, 13, 2, 7, \\ & 10, 0, 5, 12, 15, 3, 8, 1, 6, 11, 14, 4, 9, 16, 2, 7, 0, 5, 10, 15, 3, 8, 13, 1, 6, 11, 4, 9, \\ &  14, 2, 7, 12.\\

\end{array}
\end{equation*}

\end{observ}

\begin{proof}
All the patterns were obtained and verified by a computer search. Obviously, in each of the pattern, the maximum value is $16$ while the minimum value is $0$, hence $\lambda_{(3,2,1)}(G)\leq 16$ for every $n$ listed in the assumptions.
\end{proof}

\begin{theorem}
\label{THM C14 upper}
Let $n\in \mathbb{N}$ and $G=C_n(1,4)$. If $n\geq 57$, then
$\lambda_{(3,2,1)}(G) \leq 16.$
\end{theorem}

\begin{proof}
We use the pattern from the proof of Theorem $\ref{THM C14 presne}$, i.e., pattern $P=(0,5,10,15,3,8,13,1,6,11,4,9,14,2,7,12)$. In the proof, we will distinguish possibilities depending on the remainder $r$ after division of $n$ by $16$. Let $n=16k+g(r)$ for some $k\in\mathbb{N}\cup\{0\}$, where $g(0)=0$, $g(1)=17$, $g(2)=34$, $g(3)=51$, $g(4)=68$, $g(5)=85$, $g(6)=86$, $g(7)=87$, $g(8)=72$, $g(9)=57$, $g(10)=42$, $g(11)=59$, $g(12)=76$, $g(13)=61$, $g(14)=46$ and $g(15)=31$. In each of the cases, we label vertices of $G$ in such a way that we copy pattern $P$ $k$ times on vertices $u_1,u_2,\dots, u_{16k}$, and then, for the remaining vertices of $G$, we use the pattern given in Observation $\ref{C14 dlouhe vzory}$ for relevant $g(r)$. It follows that the only possible collision in the defined labeling of $G$ could be between some vertex labeled by pattern $P$ and some vertex labeled by the pattern of length $g(r)$. Specifically, it suffices to show that there is no collision between any vertex from $\{u_1,u_2,\dots, \, u_{12}, u_{16k-11}, u_{16k-10}, \dots, \, u_{16k}\}$ and any vertex from $\{u_{16k+1}, u_{16k+2},\dots, \, u_{16k+12}, u_{n-11}, u_{n-10}, \dots, \, u_{n}\}$, since for every $u_i,u_j\in V(G)$ with $\dist_G(u_i,u_j)\leq 3$ we have $|j-i|\leq 12$ or $|j-i|\geq n-12$.

Hence we distinguish the following cases depending on the remainder $r$ after division of $n$ by $16$. We have the following sequences of labels (in the first line, there is pattern $P$ followed by the pattern for $n'=g(r)$; in the second line, the patterns are switched) for $r\in\{0,\dots,15\}$. Note that the symbol $|$ means the border between the patterns.

\begin{enumerate}

\item[{\bf Case 1:}\hspace*{-0.7cm}] \hspace{0.7cm}{\sl $r=0$, i.e. $n'=16$}.

It follows from Theorem $\ref{THM C14 presne}$.

\item[{\bf Case 2:}\hspace*{-0.7cm}] \hspace{0.7cm}{\sl $r\in\{1,2,3,4,5\}$, i.e. $n'\in\{17,34,51,68,85\}$}.

\begin{equation*}
\dots,3,8,13,1,6,11,4,9,14,2,7,12, \,\,\vert \,\, 0, 5, 10, 15, 3, 8, 13, 1, 6, 11, 16, 4,  \dots 
\end{equation*}
\begin{equation*}
\dots,8, 13, 1, 6, 11, 16, 4, 9, 14, 2, 7, 12, \,\, \vert \,\, 0,5,10,15,3,8,13,1,6,11,4,9, \dots  
\end{equation*}

\item[{\bf Case 3:}\hspace*{-0.7cm}] \hspace{0.7cm}{\sl $r=6$, i.e. $n'=86$}.
\begin{equation*}
\dots,3,8,13,1,6,11,4,9,14,2,7,12, \,\,\vert \,\, 0, 5, 10, 16, 3, 8, 14, 1, 6, 12, 4, 9,  \dots 
\end{equation*}
\begin{equation*}
\dots,3, 8, 13, 1, 6, 11, 4, 9, 14, 2, 7, 12, \,\, \vert \,\, 0,5,10,15,3,8,13,1,6,11,4,9, \dots  
\end{equation*}

\item[{\bf Case 4:}\hspace*{-0.7cm}] \hspace{0.7cm}{\sl $r=7$, i.e. $n'=87$}.
\begin{equation*}
\dots,3,8,13,1,6,11,4,9,14,2,7,12, \,\,\vert \,\, 0, 5, 10, 15, 8, 13, 1, 4, 11, 16, 6, 9,  \dots 
\end{equation*}
\begin{equation*}
\dots,3, 8, 13, 1, 6, 11, 4, 9, 14, 2, 7, 12, \,\, \vert \,\, 0,5,10,15,3,8,13,1,6,11,4,9, \dots  
\end{equation*}

\item[{\bf Case 5:}\hspace*{-0.7cm}] \hspace{0.7cm}{\sl $r\in\{8,9\}$, i.e. $n'=72$ or $n'=57$, respectively}.
\begin{equation*}
\dots,3,8,13,1,6,11,4,9,14,2,7,12, \,\,\vert \,\, 0, 5, 10, 15, 3, 8, 13, 1, 6, 11, 4, 9,  \dots 
\end{equation*}
\begin{equation*}
\dots,15, 8, 0, 13, 6, 11, 16, 4, 9, 14, 2, 7, \,\, \vert \,\, 0,5,10,15,3,8,13,1,6,11,4,9, \dots  
\end{equation*}

\item[{\bf Case 6:}\hspace*{-0.7cm}] \hspace{0.7cm}{\sl $r=10$, i.e. $n'=42$}.
\begin{equation*}
\dots,3,8,13,1,6,11,4,9,14,2,7,12, \,\,\vert \,\, 0, 5, 10, 3, 16, 8, 1, 14, 6, 12, 4, 9,  \dots 
\end{equation*}
\begin{equation*}
\dots,8, 0, 13, 6, 11, 16, 4, 9, 14, 2, 7, 12, \,\, \vert \,\, 0,5,10,15,3,8,13,1,6,11,4,9, \dots  
\end{equation*}

\item[{\bf Case 7:}\hspace*{-0.7cm}] \hspace{0.7cm}{\sl $r=11$, i.e. $n'=59$}.
\begin{equation*}
\dots,3,8,13,1,6,11,4,9,14,2,7,12, \,\,\vert \,\, 0, 5, 10, 3, 16, 8, 1, 14, 6, 12, 4, 9,  \dots 
\end{equation*}
\begin{equation*}
\dots,8, 13, 1, 6, 11, 16, 4, 9, 14, 2, 7, 12, \,\, \vert \,\, 0,5,10,15,3,8,13,1,6,11,4,9, \dots  
\end{equation*}

\item[{\bf Case 8:}\hspace*{-0.7cm}] \hspace{0.7cm}{\sl $r\in\{12,13\}$, i.e. $n'=76$ or $n'=61$, respectively.}
\begin{equation*}
\dots,3,8,13,1,6,11,4,9,14,2,7,12, \,\,\vert \,\, 0, 5, 10, 15, 3, 8, 13, 1, 6, 11, 4, 9,  \dots 
\end{equation*}
\begin{equation*}
\dots,15, 3, 8, 13, 6, 11, 16, 4, 9, 14, 2, 7, \,\, \vert \,\, 0,5,10,15,3,8,13,1,6,11,4,9, \dots  
\end{equation*}

\item[{\bf Case 9:}\hspace*{-0.7cm}] \hspace{0.7cm}{\sl $r=14$, i.e. $n'=46$}.
\begin{equation*}
\dots,3,8,13,1,6,11,4,9,14,2,7,12, \,\,\vert \,\, 0, 5, 10, 15, 3, 8, 13, 1, 6, 11, 4, 9,  \dots 
\end{equation*}
\begin{equation*}
\dots,15, 3, 8, 1, 6, 11, 14, 4, 9, 16, 2, 7, \,\, \vert \,\, 0,5,10,15,3,8,13,1,6,11,4,9, \dots  
\end{equation*}

\item[{\bf Case 10:}\hspace*{-0.7cm}] \hspace{0.7cm}{\sl $r=15$, i.e. $n'=31$}.
\begin{equation*}
\dots,3,8,13,1,6,11,4,9,14,2,7,12, \,\,\vert \,\, 0, 5, 10, 15, 3, 8, 13, 1, 6, 11, 4, 9,  \dots 
\end{equation*}
\begin{equation*}
\dots,3, 8, 13, 1, 6, 11, 16, 4, 9, 2, 7, 12, \,\, \vert \,\, 0,5,10,15,3,8,13,1,6,11,4,9, \dots  
\end{equation*}

\end{enumerate}

In each of the cases, two vertices of $G$ are adjacent if and only if two numbers in any of the above listed sequences are next to each other or they have exactly three numbers in-between.
Also, two vertices of $G$ are at distance $2$ if and only if two numbers in any of the above listed sequences have exactly one, two, four or seven numbers in-between.
And, two vertices of $G$ are at distance $3$ if and only if the numbers in the listed sequences have exactly five, six, eight or eleven numbers in-between.

Clearly, in every sequence, labels of each pair of adjacent vertices differ by at least $3$, and labels of each pair of vertices at distance $2$ apart differ by at least $2$. For each $u,v\in V(G)$ with $\dist_G(u,v)=3$, labels of $u$ and $v$ are distinct since any pair of the same value in each of the sequences has $9, 10$ or more than $12$ other values in-between. Therefore the defined labeling is an $L(3,2,1)$-labeling of $G$ in each of the cases. Since we used all remainders after division by sixteen, we showed that statement holds for all $n\geq 57$ except $n\in\{60,69,70,71\}$. But together with Observation $\ref{C14 dlouhe vzory vyjimka}$ 
we have $\lambda_{(3,2,1)}(G)\leq 16$ for all $n\geq 57$. 
\end{proof}

Analogously as for $C_n(1,3)$, we tested all $C_n(1,4)$ for $n\leq 500$ and we observed that $\lambda_{(3,2,1)}(C_n(1,4))=15$ if and only if $n$ is divisible by $16$. Thus we propose the following conjecture.

\begin{conjecture}
Let $n\in \mathbb{N}$ and $G=C_n(1,4)$. If $n\geq 57$ and is not divisible by $16$, then
$\lambda_{(3,2,1)}(G)=16$.
\end{conjecture}

We finish this section by summarizing results for $C_n(1,4)$ shown in Proposition \ref{PROP C14 lower} and Theorems \ref{THM C14 presne} and \ref{THM C14 upper}.

\begin{corollary}
Let $n\in \mathbb{N}$, 
and let $G=C_n(1,4)$. Then
\begin{equation*}
\begin{array}{rl}\lambda_{(3,2,1)}(G) =15 & \mbox{when } n \mbox{ is divisible by } 16, \\
15\leq \lambda_{(3,2,1)}(G) \leq 16 & \mbox{when } n\geq 57. 
\end{array}
\end{equation*}
\end{corollary}

\section{Circulants $C_n(1,5)$}

Obviously, $n\geq 11$. First we give a lower bound on $\lambda_{(3,2,1)}(G)$.

\begin{proposition}
\label{PROP C15 lower}
Let $n \in \mathbb{N}$, $n \geq 11$, and $G=C_n(1,5)$. Then $\lambda_{(3,2,1)}(G)\geq 13$.
\end{proposition}

\begin{proof}
For the distance graph $G(1,5)$, using a computer check we found that $\lambda_{(3,2,1)}(G(1,5)) \geq 13$. We tried to label vertices of $G(1,5)$ with labels $0,1,\dots, 12$, but the programme was able to label only $29$ consecutive vertices of $G(1,5)$ using these labels. The programme began by assigning label $0$ to a fixed vertex $i$ of $G(1,5)$ and then attempted to extend this labeling to vertices $i+1,i+2,i+3,\dots$ Note that such a vertex $i$ must exist, since otherwise all labels could be decreased uniformly, reducing the span. 

Let $G_1$ denote a subgraph of $G(1,5)$ induced by vertices $1,2,\dots, 30$. We know that $G_1$ has no $L(3,2,1)$-labeling using labels $0,1,\dots, 12$ such that vertex $1$ has label $0$. Since every $C_n(1,5)$ ($n>29$) contains $G_1$ as a subgraph, we have $\lambda_{(3,2,1)}(G)\geq 13$ for every $n>29$. From Tables $\ref{table_1_5_even}$ and $\ref{table_1_5_odd}$, the lower bound $13$ is valid also for every $n \in \{11,\dots,29\}$.
\end{proof}

First we deal with $n$ even. In Table $\ref{table_1_5_even}$, we list exact values of $\lambda_{(3,2,1)}(C_n(1,5))$ for even $n\leq 58$. These exact values were obtained by computer experiments. We determined that the span of feasible labelings of such graphs cannot be reduced, and the reported values were confirmed through the identification and verification of relevant patterns (see Appendix).

\begin{table}
\centering
\begin{tabular}{|c|c||c|c||c|c|}

  \hline
  $n$ & $\lambda_{(3,2,1)}(G)$ & $n$ & $\lambda_{(3,2,1)}(G)$ & $n$ & $\lambda_{(3,2,1)}(G)$\\
  \hline \hline
  12 & 13 & 28 & 13 & 44 & 13\\
  14 & 13 & 30 & 13 & 46 & 13\\
  16 & 15 & 32 & 14 & 48 & 13\\
  18 & 17 & 34 & 13 & 50 & 13\\
  20 & 13 & 36 & 13 & 52 & 13\\
  22 & 13 & 38 & 13 & 54 & 13\\
  24 & 13 & 40 & 13 & 56 & 13\\
  26 & 13 & 42 & 13 & 58 & 13\\
  \hline
\end{tabular}
\caption{The exact values of $\lambda_{(3,2,1)}(G)$ for circulant $G=C_n(1,5)$ with even $n\leq 58$}
\label{table_1_5_even}
\end{table}

The following theorem gives exact values on $\lambda_{(3,2,1)}$ of $C_n(1,5)$ with even $n$.

\begin{theorem}
\label{THM C15 even}
Let $n \in \mathbb{N}$, $n \geq 34$ even, and $G=C_n(1,5)$. Then $\lambda_{(3,2,1)}(G) =13$.
\end{theorem}

\begin{proof}
For $n\leq 58$, the statement follows from Table $\ref{table_1_5_even}$. Hence we assume that $n\geq 60$. By the Sylvester's theorem (Theorem \ref{sylvester}), for every $n'\geq 30$, there are nonnegative integers $k',\ell'$ such that $n'=6k'+7\ell'$. Therefore, for every even integer $n\geq 60$, there are positive integers $k,\ell$ such that $n=12k+14\ell$. 
We use the following two patterns, one of length $12$ (denoted by $P_1$) and one of length $14$ (denoted by $P_2$), both using values from $\{0,1,2,\dots, 13\}$ (see Fig. \ref{FIG C15}):

\begin{equation*}
P_1=(0,5,10,1,6,11,2,7,12,3,8,13); \hfill
P_2=(0,5,10,1,6,11,2,7,12,3,8,13,4,9).
\end{equation*}

\begin{figure}[ht]
  \centering
  \includegraphics[width=0.80\textwidth]{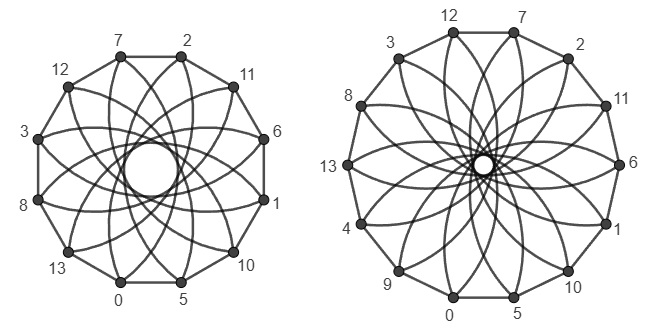}
  \caption{Two circulants $C_{12}(1,5)$ and $C_{14}(1,5)$ labeled using the values from $0$ to $13$.}
  \label{FIG C15}
\end{figure}

We define a labeling $f$ of the vertices of $G$ by repeating pattern $P_1$ $k$ times on the vertices $u_1, u_2, \dots, u_{12k}$, and then repeating pattern $P_2$ $\ell$ times on the vertices $u_{12k+1}, u_{12k+2}, \dots, u_n$.
First, we show that there is no collision between any pair of vertices among $u_1, \dots, u_{12k}$. Since no label is repeated within $P_1$ and since the indices of vertices among $u_1, \dots, u_{12k}$ that share the same label differ by a multiple of $12$, the distance between any pair of vertices from $u_1, \dots, u_{12k}$ with the same label is at least $4$. Therefore it suffices to verify the labeling between two consecutive copies of $P_1$. This yields the following sequence of labels (where the symbol $\vert$ denotes the boundary between the two copies of $P_1$):

\begin{equation*}
0,5,10,1,6,11,2,7,12,3,8,13 \;\vert\; 0,5,10,1,6,11,2,7,12,3,8,13.
\end{equation*}

It is straightforward to verify that this labeling satisfies all distance constraints of an $L(3,2,1)$-labeling. By analogous arguments, we conclude that there is no collision between any pair of vertices among $u_{12k+1}, u_{12k+2}, \dots, u_n$, where we obtain the following sequence of labels:

\begin{equation*}
0,5,10,1,6,11,2,7,12,3,8,13,4,9 \;\vert\; 0,5,10,1,6,11,2,7,12,3,8,13,4,9.
\end{equation*}

It remains to show that the patterns $P_1$ and $P_2$ can be combined into a feasible labeling of $G$. Again, since no label is repeated within $P_1$ and within $P_2$ and the indices of vertices sharing the same label differ by $12$ or $14$ or at least $24$, the distance between any pair of vertices of $G$ with the same label is at least $4$. Therefore it suffices to show that there is no collision between two consecutive patterns of distinct lengths. We have the following sequences of labels (in the first line, $P_1$ is followed by $P_2$; in the second line, the patterns are switched). Note that the symbol $\vert$ denotes the boundary between the patterns.

\begin{equation*}
0,5,10,1,6,11,2,7,12,3,8,13, \,\, \vert \,\, 0,5,10,1,6,11,2,7,12,3,8,13,4,9,
\end{equation*}

\begin{equation*}
0,5,10,1,6,11,2,7,12,3,8,13,4,9, \,\, \vert \,\, 0,5,10,1,6,11,2,7,12,3,8,13.
\end{equation*}

Obviously, $P_1$ is a subsequence of $P_2$, therefore we infer that there is no collision in the labeling $f$, and hence $f$ is an $L(3,2,1)$-labeling of $G$.
\end{proof}

Now we move on to odd $n$. In Table $\ref{table_1_5_odd}$, we list exact values of $\lambda_{(3,2,1)}(C_n(1,5))$ for odd $n\leq 45$. These values were obtained by a computer search. Analogously as for the previous tables, the programme verified that the values of $\lambda_{(3,2,1)}$ cannot be reduced by showing that the corresponding graph admits no labeling using values $0,\dots, \lambda_{(3,2,1)}-1$. On the other hand, for each $n$ listed in the table it gave a feasible labeling by a repetitive pattern which is presented Appendix. Due to an increasing computational time required for larger $n$ (exceeding one day) we were not able to find exact values for larger $n$. However, we prove the following general upper bound on $\lambda_{(3,2,1)}$  for $C_n(1,5)$.

\begin{table}

\centering
\begin{tabular}{|c|c||c|c||c|c|}

  \hline
  $n$ & $\lambda_{(3,2,1)}(G)$ & $n$ & $\lambda_{(3,2,1)}(G)$ & $n$ & $\lambda_{(3,2,1)}(G)$\\
  \hline \hline
  11 & 15 & 23 & 22 & 35 & 16 \\
  13 & 24 & 25 & 16 & 37 & 16 \\
  15 & 14 & 27 & 17 & 39 & 16 \\
  17 & 16 & 29 & 16 & 41 & 16 \\
  19 & 18 & 31 & 15 & 43 & 15 \\
  21 & 20 & 33 & 16 & 45 & 14 \\
  \hline
\end{tabular}
\caption{The exact values of $\lambda_{(3,2,1)}(G)$ for circulant $G=C_n(1,5)$ with odd $n\leq 45$}
\label{table_1_5_odd}
\end{table}

\begin{theorem} \label{THM C15 odd}
Let $n \in \mathbb{N}$, $n \geq 105$ odd, and $G=C_n(1,5)$. Then

\begin{equation*}
\lambda_{(3,2,1)}(G) \leq 14. 
\end{equation*}
\end{theorem}

\begin{proof}
Let $n\geq 105$ be odd. From Theorem \ref{sylvester} and from the proof of Theorem \ref{THM C15 even} we know that, for every even integer $n'\geq 60$, there are two nonnegative integers $k',\ell'$ such that $n'=12k'+14\ell'$ and $\lambda_{(3,2,1)}(C_{n'}(1,5))=13$ for these $n'$. It follows that, for every odd integer $n\geq 105$, there are two nonnegative integers $k,\ell$ such that $n=45+12k+14\ell$. From Table $\ref{table_1_5_odd}$ we know that $\lambda_{(3,2,1)}(C_{45}(1,5))=14$. We use the two patterns presented in the proof of Theorem \ref{THM C15 even} - the patterns $P_1$ and $P_2$, and the pattern for a labeling of $C_{45}(1,5)$ (the appropriate pattern is presented in Appendix); we denote this pattern by $P_3$. We show that these patterns can be combined into a feasible labeling of $G$.

We define a labeling $f$ of the vertices of $G$ in such a way that we use pattern $P_3$ on vertices $u_1,u_2,\dots, \, u_{45}$, then we copy pattern $P_1$ $k$ times on vertices $u_{46},u_{47},\dots, \, u_{12k+45}$ and then we copy pattern $P_2$ $\ell$ times on vertices $u_{12k+46},$ $u_{12k+47},
\dots, \, u_n$. From the proof of Theorem \ref{THM C15 even}, we know that there is no collision in the defined labeling between any pair of vertices among $u_{46},u_{47},\dots, u_n$. It remains to show that we can combine pattern $P_3$ with patterns $P_1$ and $P_2$ into a feasible labeling of $G$. It suffices to show that there is no collision between two consecutive patterns of these lengths.

We have the following sequences of labels (in the first line, pattern $P_3$ is followed by pattern $P_1$, in the second line, the patterns are switched; in the third line, pattern $P_3$ is followed by pattern $P_2$, in the fourth line, the patterns are switched). Note that the symbol $\vert$ means the boundary between the patterns and we use only first or last fifteen values of $P_3$, since for every $u_i,u_j\in V(G)$ with $\dist_G(u_i,u_j)=3$ we have $|j-i|\leq 15$ or $|j-i|\geq n-15$.

\begin{equation*}
...,10,0,5,11,1,6,12,2,7,13,3,8,14,4,9,\,\, \vert \,\, 0,5,10,1,6,11,2,7,12,3,8,13,\dots,
\end{equation*}

\begin{equation*}
0,5,10,1,6,11,2,7,12,3,8,13, \,\, \vert \,\, 0,5,10,1,6,11,2,7,12,3,8,13,4,9,14,...,
\end{equation*}

\begin{equation*}
...,10,0,5,11,1,6,12,2,7,13,3,8,14,4,9,\,\, \vert \,\, 0,5,10,1,6,11,2,7,12,3,8,13,4,9,
\end{equation*}

\begin{equation*}
0,5,10,1,6,11,2,7,12,3,8,13,4,9, \,\, \vert \,\, 0,5,10,1,6,11,2,7,12,3,8,13,4,9,14,...
\end{equation*}

Obviously, patterns $P_1$ and $P_2$ are subsequences of pattern $P_3$ and hence the same values (labels) are pairwise at appropriate distances. Therefore we infer that there is no collision in the labeling $f$, and hence $f$ is an $L(3,2,1)$-labeling of $G$.
\end{proof}

Using a computer, we tested circulants $C_n(1,5)$ for all odd $n\leq 151$ and we found that no such circulant can be labeled using labels $0,1,2,\dots,13$. Thus we believe that the following conjecture is true.

\begin{conjecture}
Let $n\in \mathbb{N}$, $n\geq 105$ odd, and let $G=C_n(1,5)$. Then

\begin{equation*}
\lambda_{(3,2,1)}(G) = 14. 
\end{equation*}
\end{conjecture}

We finish this section by summarizing of results for $C_n(1,5)$ proved in Proposition \ref{PROP C15 lower} and Theorems \ref{THM C15 even} and \ref{THM C15 odd}.

\begin{corollary}
Let $n\in \mathbb{N}$, 
and let $G=C_n(1,5)$. Then
\begin{equation*}
\begin{array}{rl}\lambda_{(3,2,1)}(G) =13 & \mbox{when } n\geq 34 \mbox{ is even}, \\
13\leq \lambda_{(3,2,1)}(G) \leq 14 & \mbox{when } n\geq 105 \mbox{ is odd}. 
\end{array}
\end{equation*}
\end{corollary}

\section{Computer computations}

For our computations, we used a simple backtracking algorithm with some preprocessing to establish some lower bounds, exact values listed in the tables, and also for obtaining most of the patterns for $L(3,2,1)$-labelings. With some small modifications, we also used this algorithm for verification of correctness of some patterns obtained "by hand". 

The main programme takes the following input parameters: integer $n$, set of integers $S$ of the graph $C_n(S)$, and a number of used labels $k$. The set of vertices of $C_n(S)$ is represented by a vector (array) of length $n$. We usually set the label of vertex "1" to 0.  Then the algorithm tries to extend such a labeling to the rest of $G$. Since circulants have a cyclic structure, it is necessary to check possible collisions counting modulo $n$. If the algorithm fails, we observe that $\lambda_{(3,2,1)}(C_n(S))>k$. On the other hand, if the algorithm finds an $L(3,2,1)$-labeling, we try to find a repetition of labels in a labeling in order to obtain a repetitive pattern. For some such experiments (e.g. in the proof of Theorem \ref{THM 13 odd}) we fixed labels of some  vertices (usually first 10 vertices) and then we searched for patterns starting with the given sequence of labels. In addition, we verified the patterns using some small modification of the main algorithm.

All the computations were performed on one core of 4-core CPU \emph {Intel(R) Core(TM) i$7$ $2.3$ GHz} with $16$GB RAM. We implemented two programmes, one in PASCAL and one in JAVA. Both variations - verification of the existence of a feasible labeling with given span, and verification of correctness of a pattern -  are available on the website \\ \centerline{\url{https://www.iti.zcu.cz/Holub/L3,2,1-labelingsofcirculants/}.}

\medskip

In addition, we verified correctness of exact values of $\lambda_{(3,2,1)}$ presented in Tables \ref{table_1_3}, \ref{table 13 odd}, \ref{table_1_4}, \ref{table_1_5_even} and \ref{table_1_5_odd} using a MILP solver. On the above mentioned website, we provide a generator of MILP codes for computing the exact values of $\lambda_{(3,2,1)}(C_n(S))$. The generated code can be downloaded and submitted to a MILP solver; the website also contains a link to the NEOS Server with the Gurobi optimiser. Below we provide the MILP code for $C_n(1,t)$. 

\clearpage 

\begin{algorithm}[H]
\label{alg}
\caption{MILP model for the $L(3,2,1)$-labeling of the circulant graph $C_n(1,t)$}
\begin{algorithmic}[1]

\Require Graph $C_n(1,t)$ with $V=\{0,\dots,n-1\}$
\Ensure Optimal $L(3,2,1)$-labeling

\Statex
\State \textbf{Variables}
\For{$i \in V$}
    \State $f_i \in \mathbb{Z}$
\EndFor
\State $M \in \mathbb{Z}$
\For{$i,j \in V,\ i<j$}
    \State $b_{ij} \in \{0,1\}$
\EndFor

\Statex
\State \textbf{Objective function}
\State Minimize $M$

\Statex
\State \textbf{Feasibility constraints}
\For{$i \in V$}
    \State $f_i \le M$
\EndFor

\For{$i,j \in V,\ i<j$}
    \State $f_j - f_i + 25\, b_{ij} \le 25 - \delta(d(i,j))$
    \State $f_i - f_j - 25\, b_{ij} \le -\delta(d(i,j))$
\EndFor

\Statex
\State \textbf{Constraints on label values}
\For{$i \in V$}
    \State $0 \le f_i \le 25$
\EndFor
\State $0 \le M$
\For{$i,j \in V,\ i<j$}
    \State $0 \le b_{ij} \le 1$
\EndFor

\end{algorithmic}
\end{algorithm}

We use the distance--dependent separation function
\[
\delta(d)=
\begin{cases}
3, & d=1,\\
2, & d=2,\\
1, & d=3,\\
\end{cases}
\]
corresponding to the $L(3,2,1)$--labeling constraints.

\paragraph{Objective function}
The objective function minimizes the maximum label $M$ assigned to the
vertices of the graph. This corresponds to minimizing the span of the
$L(3,2,1)$--labeling and ensures that an optimal labeling with respect
to the maximum label value is obtained.

\paragraph{Feasibility constraints between vertices}
The feasibility of a labeling is enforced by distance--based separation
constraints. For each unordered pair of vertices $(i,j)$, the difference
between their assigned labels must satisfy the $L(3,2,1)$ conditions,
depending on the graph distance $d(i,j)$ in the circulant graph $C_n(1,t)$.
Vertices at distance~$1$ are required to differ by at least~$3$, while
vertices at distance~$2$ must differ by at least~$2$.

\paragraph{Constraints on label values}
Bounds on the label variables restrict each vertex label $f_i$ to lie
between $0$ and the predefined upper bound $\lambda_{\max}=25$. These
constraints reduce the feasible search space and reflect the assumption
that an optimal labeling exists within this range.

\paragraph{Binary variables}
Auxiliary binary variables $b_{ij}$ are introduced to linearize the
absolute value constraints arising from label separation requirements.
Each binary variable determines which of the two corresponding linear
inequalities is active, resulting in a mixed--integer linear programming
formulation.

\bigskip

\bigskip





\clearpage 
\section*{Appendix - patterns}
Here we list patterns for sparse circulants mentioned in Tables 
\ref{table_1_3}, \ref{table 13 odd}, \ref{table_1_4}, \ref{table_1_5_even} and \ref{table_1_5_odd}.
%
%
%
%
%
%
%
%
%


$\mathbf{C_n(1,3)}$
\scriptsize
\bigskip

\begin{tabular}{|r|r|l|}
\hline
$n$ & $\lambda_{(3,2,1)}$ & pattern \\
\hline 
$7$ & $12$ & 0,6,12,4,10,2,8, \\
$8$ & $15$ & 0,9,2,11,4,13,6,15, \\
$9$ & $16$ & 0,4,8,12,16,2,6,10,14,\\
$10$ & $13$ & 0,3,6,9,12,1,4,7,10,13,\\
$11$ & $20$ & 0,4,8,12,16,2,6,10,18,14,20, \\
$12$ & $11$ & 0,5,10,3,8,1,6,11,4,9,2,7, \\
$13$ & $18$ & 0,3,6,9,12,1,4,15,18,13,7,10,16, \\
$14$ & $13$ & 0,3,6,9,12,1,4,7,10,13,2,5,8,11, \\
$15$ & $14$ & 0,5,10,3,8,13,6,11,1,9,14,4,12,2,7, \\
$16$ & $15$ & 0,3,6,9,12,1,4,7,14,11,2,5,8,13,10,15, \\
$17$ & $16$ & 0,3,6,9,12,1,14,7,10,4,16,13,8,2,5,15,11, \\
$18$ & $13$ & 0,3,6,13,10,1,8,5,12,3,0,7,10,13,2,5,8,11, \\
$19$ & $18$ & 0,3,6,9,12,1,4,7,10,15,18,13,2,5,16,11,14,8,18, \\
$20$ & $13$ & (0,3,6,9,12,1,4,7,10,13)$^2$, \\
$21$ & $16$ & 0,3,6,9,12,15,4,7,1,11,16,13,6,0,9,2,14,4,16,8,11, \\
$22$ & $13$ & 0,3,6,11,8,13,2,5,0,9,12,3,6,1,8,13,10,5,2,7,12,9, \\
$23$ & $15$ & 0,3,8,13,6,11,4,9,2,15,0,13,6,11,4,9,14,7,12,2,10,15,5, \\
$24$ & $11$ & (0,5,10,3,8,1,6,11,4,9,2,7)$^2$, \\
$25$ & $14$ & 0,7,12,3,10,1,6,13,4,9,2,7,12,5,10,0,8,13,3,11,1,6,14,4,9, \\
$26$ & $13$ & 0,3,6,11,8,1,4,13,10,7,0,5,2,11,8,13,6,1,4,9,12,7,2,5,10,13, \\
$27$ & $14$ & 0,3,8,11,6,1,4,9,12,7,14,5,10,0,8,13,3,11,1,6,14,4,9,2,7,12,5, \\
$28$ & $13$ & (0,3,6,9,12,1,4,7,10,13,2,5,8,11)$^2$, \\
$29$ & $14$ & 0,3,8,5,10,1,12,7,14,9,4,11,6,1,8,3,13,5,0,10,2,12,7,14,9,4,11,\\ & & 6,13, \\
$30$ & $13$ & (0,3,6,9,12,1,4,7,10,13)$^3$, \\
$31$ & $15$ & 0,3,6,9,12,1,4,7,10,13,0,3,6,15,8,11,2,5,14,7,12,0,10,15,3,13,\\ & & 1,11,5,8,14, \\
$32$ & $13$ & 0,3,6,9,12,1,4,7,10,13,0,3,8,5,12,1,10,7,4,13,0,9,2,5,12,7,10,\\ & & 1,4,13,8,11, \\
$33$ & $15$ & 0,3,6,9,12,1,4,7,10,13,0,3,8,5,15,1,11,7,4,14,0,12,9,3,15,6,13,\\ & & 10,1,8,5,14,11, \\
$34$ & $13$ & (0,3,6,9,12,1,4,7,10,13)$^3$,2,5,8,11, \\
$35$ & $14$ & 0,3,6,9,12,1,4,7,10,13,2,5,8,11,14,0,6,9,3,13,1,11,8,4,14,0,10,\\ & & 6,3,13,1,11,5,8,14, \\
$36$ & $11$ & (0,5,10,3,8,1,6,11,4,9,2,7)$^3$, \\
$37$ & $14$ & 0,3,6,9,12,1,4,7,10,13,2,5,8,11,14,0,6,9,3,13,1,11,8,4,14,2,12,\\ & & 0,10,6,3,13,1,11,5,8,14, \\
$38$ & $13$ & 0,3,6,9,12,1,4,7,10,13,(0,3,6,9,12,1,4,7,10,13,2,5,8,11)$^2$, \\
$39$ & $14$ & 0,3,6,9,12,1,4,7,10,13,0,5,2,9,12,7,14,1,10,4,8,13,0,11,3,9,14,\\ & & 6,12,2,10,4,7,13,1,11,5,8,14, \\
$40$ & $13$ & (0,3,6,9,12,1,4,7,10,13)$^4$, \\
$41$ & $13$ & 0,3,8,11,6,1,4,9,12,7,0,5,10,3,8,13,6,11,1,9,4,7,12,0,10,2,8,\\ & & 13,5,11,3,9,1,6,12,4,10,2,7,13,5, \\
$42$ & $13$ & (0,3,6,9,12,1,4,7,10,13)$^2$,0,3,8,5,12,1,10,7,4,13,0,9,2,5,12,7,\\ & & 10,1,4,13,8,11, \\
$43$ & $14$ & 0,3,6,9,12,1,4,7,10,13,0,3,6,9,14,11,4,7,1,13,10,5,8,0,14,11,3,\\ & & 7,1,9,12,4,14,0,10,6,3,13,1,11,5,8,14,  \\
\hline
\end{tabular}

\begin{tabular}{|r|r|l|}
\hline
$n$ & $\lambda_{(3,2,1)}$ & pattern \\
\hline 
$44$ & $13$ & (0,3,6,9,12,1,4,7,10,13)$^3$,2,5,8,11, \\
$45$ & $14$ & (0,3,6,9,12,1,4,7,10,13)$^2$,2,5,8,11,14,0,6,9,3,13,1,11,8,4,14,\\ & & 0,10,6,3,13,1,11,5,8,14, \\
$46$ & $13$ & (0,3,6,9,12,1,4,7,10,13)$^2$,0,3,6,11,8,1,4,13,10,7,0,5,2,11,8,\\ & & 13,6,1,4,9,12,7,2,5,10,13, \\
$47$ & $14$ & (0,3,6,9,12,1,4,7,10,13)$^2$,2,5,8,11,14,0,6,9,3,13,1,11,8,4,14, \\ & &  2,12,0,10,6,3,13,1,11,5,8,14, \\
$48$ & $11$ & (0,5,10,3,8,1,6,11,4,9,2,7)$^4$, \\
$49$ & $14$ & (0,3,6,9,12,1,4,7,10,13)$^2$,0,5,2,9,12,7,14,1,10,4,8,13,0,11,3, \\ & & 9,14,6,12,2,10,4,7,13,1,11,5,8,14, \\
$50$ & $12$ & 0,3,8,5,10,1,12,7,4,9,0,11,6,3,8,1,10,5,12,7,0,9,4,11,6,1,8,3, \\ & & 10,5,0,7,12,9,4,1,6,11,8,3,0,5,10,7,12,1,4,9,6,11. \\
\hline
\end{tabular}

\normalsize

\bigskip

\mbox{}

\bigskip
$\mathbf{C_n(1,4)}$
\scriptsize
\bigskip

\renewcommand{\baselinestretch}{1.3}

\begin{tabular}{|c|c|l|}
\hline
$n$ & $\lambda_{(3,2,1)}$ & pattern \\
\hline 
$9$ & $16$ & 0,4,8,2,14,10,16,12,6; \\
$10$ & $18$ & 0,4,8,2,6,10,14,18,12,16; \\
$11$ & $20$ & 0,4,8,2,6,10,14,18,12,16,20; \\
$12$ & $16$ & 0,3,6,9,12,15,1,4,7,10,13,16; \\
$13$ & $18$ & 0,3,6,9,12,15,1,4,7,10,18,13,16; \\
$14$ & $19$ & 0,3,6,9,12,15,18,1,4,7,10,13,16,19; \\
$15$ & $16$ & 0,6,12,2,8,14,4,10,16,1,7,13,3,9,15; \\
$16$ & $15$ & 0,5,10,15,3,8,13,1,6,11,4,9,14,2,7,12; \\
$17$ & $16$ & 0,5,10,15,3,8,13,1,6,11,16,4,9,14,2,7,12; \\
$18$ & $17$ & 0,5,10,15,3,8,13,1,17,11,6,9,14,4,16,12,2,7; \\
$19$ & $18$ & 0,3,6,15,12,17,9,1,4,7,14,11,16,2,5,8,13,18,10; \\
$20$ & $18$ & (0,4,8,2,6,10,14,18,12,16)$^2$; \\
$21$ & $20$ & 0,7,14,4,11,18,1,8,15,5,12,19,2,9,16,6,13,20,3,10,17; \\
$22$ & $20$ & (0,8,16,2,10,18,4,12,20,6,14)$^2$; \\
$23$ & $18$ & 0,3,6,9,16,13,1,4,11,18,8,14,6,0,16,2,9,4,13,7,15,11,18; \\
$24$ & $18$ & 0,3,6,10,13,8,15,1,18,5,11,7,3,0,17,13,10,15,8,2,5,18,12,16; \\
$25$ & $18$ & 0,3,6,14,8,18,12,2,16,10,0,5,13,7,17,11,3,15,9,1,18,5,13,16,10; \\
$26$ & $17$ & 0,5,11,3,8,14,1,6,12,17,4,10,15,2,8,13,0,6,11,17,4,9,15,2,7,13; \\
$27$ & $17$ & 0,5,10,16,3,8,14,1,6,12,17,4,10,15,2,8,13,0,6,11,16,4,9,14,2,7,12; \\
$28$ & $17$ & 0,3,6,9,16,13,1,4,7,10,17,14,2,5,12,0,9,16,3,7,13,1,10,17,5,8,15,12; \\
$29$ & $17$ & 0,5,11,3,8,14,1,6,12,17,4,9,15,0,7,13,3,11,16,1,6,14,9,4,12,17,2,7,15; \\
$30$ & $16$ & (0,6,12,2,8,14,4,10,16,1,7,13,3,9,15)$^2$; \\
$31$ & $16$ &  0,5,10,15,3,8,13,1,6,11,4,9,14,2,7,12,0,5,10,3,8,13,1,6,11,16,4,9,2,7,12; \\
$32$ & $15$ & (0,5,10,15,3,8,13,1,6,11,4,9,14,2,7,12)$^2$; \\
$33$ & $16$ & 0,5,10,15,3,8,13,1,6,11,4,9,14,2,7,12,0,5,10,15,3,8,13,1,6,11,16,4,9,14,\\ && 2,7,12; \\
$34$ & $16$ & (0,5,10,15,3,8,13,1,6,11,16,4,9,14,2,7,12)$^2$;\\
$35$ & $17$ & 0,5,10,15,3,8,13,1,6,11,16,4,9,14,2,7,12,0,5,10,15,3,8,13,1,17,11,6,9,14,\\
&& 4,16,12,2,7 \\
$36$ & $17$ & (0,5,10,15,3,8,13,1,17,11,6,9,14,4,16,12,2,7)$^2$; \\
$37$ & $18$ & 0,3,6,15,12,17,9,1,4,7,14,11,16,2,5,8,13,10,0,3,6,15,12,17,9,1,4,7,14,11,\\
&& 16,2,5,8,13,18,10; \\
$38$ & $17$ & 0,3,6,9,12,15,1,4,17,8,11,14,6,2,16,0,10,13,5,8,3,17,1,12,15,6,9,4,0,11,\\
&& 14,17,8,2,5,10,13,16; \\
$39$ & $17$ & 0,3,6,16,8,13,10,2,5,17,0,7,15,11,3,13,1,8,16,10,4,14,6,2,12,17,0,8,15,10,\\
&& 3,13,5,1,7,17,9,14,11; \\
$40$ & $17$ & 0,3,8,5,12,15,1,10,17,6,13,3,8,11,16,0,14,2,7,12,4,17,10,15,1,8,3,6,11,14,\\
&& 0,9,16,5,12,2,7,10,17,14; \\
\hline
\end{tabular}

\clearpage
\begin{tabular}{|c|c|l|}
\hline
$n$ & $\lambda_{(3,2,1)}$ & pattern \\
\hline
$41$ & $17$ & 0,3,6,11,14,17,9,2,7,12,15,5,10,0,3,8,17,14,6,1,4,9,16,13,7,2,11,5,0,17,14,\\
&& 9,3,12,7,1,15,5,10,13,8; \\
$42$ & $16$ & 0,5,10,3,16,8,1,14,6,12,4,9,2,15,7,0,13,5,11,3,16,9,1,14,7,12,4,10,2,15,8,\\
&& 0,13,6,11,16,4,9,14,2,7,12; \\
$43$ & $17$ & 0,3,6,11,14,17,9,4,7,12,1,16,10,5,14,8,3,0,17,12,6,15,10,4,1,8,13,16,11,5,\\
&& 0,9,14,17,3,12,7,1,15,5,10,13,8; \\
$44$ & $17$ & 0,3,6,9,16,13,1,4,7,10,17,14,2,5,12,0,9,16,3,6,13,1,8,15,4,10,17,12,0,6,14,\\
&& 2,9,16,4,7,13,1,10,17,5,8,15,12; \\
$45$ & $16$ & (0,6,12,2,8,14,4,10,16,1,7,13,3,9,15)$^3$; \\
$46$ & $16$ & 0,5,10,15,3,8,13,1,6,11,4,9,14,2,7,12,0,5,10,3,8,13,1,6,11,16,4,9,2,7,12,\\
&& 0,5,10,15,3,8,1,6,11,14,4,9,16,2,7; \\
$47$ & $16$ & (0,5,10,15,3,8,13,1,6,11,4,9,14,2,7,12)$^2$, 0,5,10,3,8,13,1,6,11,16,4,9,2,7,12; \\
$48$ & $15$ & (0,5,10,15,3,8,13,1,6,11,4,9,14,2,7,12)$^3$; \\
$49$ & $16$ & (0,5,10,15,3,8,13,1,6,11,4,9,14,2,7,12)$^2$, 0,5,10,15,3,8,13,1,6,11,16,4,9,\\
&& 14,2,7,12; \\
$50$ & $16$ & 0,5,10,15,3,8,13,1,6,11,4,9,14,2,7,12,\\ &&(0,5,10,15,3,8,13,1,6,11,16,4,9,14,2,7,12)$^2$; \\
$51$ & $16$ & (0,5,10,15,3,8,13,1,6,11,16,4,9,14,2,7,12)$^3$; \\
$52$ & $17$ & (0,5,11,3,8,14,1,6,12,17,4,10,15,2,8,13,0,6,11,17,4,9,15,2,7,13)$^2$; \\
$53$ & $17$ & 0,5,10,15,3,8,13,1,6,11,16,4,9,14,2,7,12,\\ && (0,5,10,15,3,8,13,1,17,11,6,9,14,4,16,12,2,7)$^2$; \\
$54$ & $17$ & (0,5,10,16,3,8,14,1,6,12,17,4,10,15,2,8,13,0,6,11,16,4,9,14,2,7,12)$^2$; \\
$55$ & $17$ & 0,5,11,3,8,14,1,6,12,17,4,10,15,2,8,13,0,6,11,17,4,9,15,2,7,13,0,5,11,3,8,\\
&& 14,1,6,12,17,4,9,15,0,7,13,3,11,16,1,6,14,9,4,12,17,2,7,15; \\
$56$ & $17$ & (0,3,6,9,16,13,1,4,7,10,17,14,2,5,12,0,9,16,3,7,13,1,10,17,5,8,15,12)$^2$. \\
\hline
\end{tabular}
\normalsize

\bigskip

\mbox{}

\bigskip

$\mathbf{C_n(1,5)}$
\scriptsize
\bigskip

\renewcommand{\baselinestretch}{1.3} 

\begin{tabular}{|c|c|l|}
\hline
$n$ & $\lambda_{(3,2,1)}$ & pattern \\
\hline 
$11$ & $15$ & 0,3,6,1,4,7,10,15,12,9,14; \\
$12$ & $13$ & 0,5,10,1,6,11,2,7,12,3,8,13; \\
$13$ & $24$ & 0,4,8,2,6,10,14,18,12,22,16,20,24; \\
$14$ & $13$ & 0,5,10,1,6,11,2,7,12,3,8,13,4,9; \\
$15$ & $14$ & 0,3,6,9,2,5,8,11,14,7,10,13,1,4,12; \\
$16$ & $15$ & 0,3,6,1,8,11,14,9,4,13,2,7,12,15,10,5; \\
$17$ & $16$ & 0,6,12,1,7,13,2,8,14,3,9,15,4,10,16,5,11; \\
$18$ & $17$ & 0,3,6,1,4,7,10,13,8,11,14,17,2,5,16,9,12,15; \\
$19$ & $18$ & 0,3,6,9,2,5,8,1,4,12,15,18,11,14,17,10,13,16,7; \\
$20$ & $13$ & 0,3,8,1,4,9,12,5,10,13,2,7,0,3,8,5,10,13,6,11; \\
$21$ & $20$ & 0,13,19,1,7,20,12,8,14,3,9,15,4,10,16,5,11,17,6,12,18; \\
$22$ & $13$ & 0,3,6,11,2,7,10,13,0,5,12,1,4,7,10,3,6,9,12,1,8,13; \\
$23$ & $22$ & 0,8,16,1,9,17,2,10,18,3,11,19,4,12,20,5,13,21,6,14,22,7,15; \\
$24$ & $13$ & (0,5,10,1,6,11,2,7,12,3,8,13)$^2$; \\
$25$ & $16$ & 0,5,11,16,4,14,2,7,13,1,10,16,4,9,15,7,12,0,6,11,3,9,14,2,8; \\
$26$ & $13$ & 0,5,10,1,6,11,2,7,12,3,8,13,0,5,10,1,6,11,2,7,12,3,8,13,4,9; \\
$27$ & $17$ & 0,3,6,9,2,5,8,11,0,7,10,13,16,4,12,17,1,7,10,15,3,11,14,17,1,13,16; \\
$28$ & $13$ & (0,5,10,1,6,11,2,7,12,3,8,13,4,9)$^2$; \\
$29$ & $16$ & 0,5,11,3,16,9,2,15,8,1,14,7,0,13,5,11,16,4,10,15,3,9,14,2,7,13,1,6,12; \\
$30$ & $13$ & 0,3,6,13,2,5,8,11,0,7,10,13,2,5,12,3,0,7,10,1,8,5,12,3,6,13,10,1,8,11; \\

\hline
\end{tabular}

\clearpage
\begin{tabular}{|c|c|l|}
\hline
$n$ & $\lambda_{(3,2,1)}$ & pattern \\
\hline
$31$ & $15$ &  0,5,11,1,6,12,2,7,13,3,8,14,4,9,15,5,10,0,6,11,1,7,12,2,8,13,3,9,14,4,10; \\
$32$ & $14$ & 0,3,6,1,12,5,10,13,4,7,14,1,8,11,2,9,12,3,6,13,0,5,10,1,4,7,2,13,8,11,14,9; \\
$33$ & $16$ & 0,3,6,9,2,5,8,15,12,7,14,11,4,0,10,3,16,9,14,5,8,0,12,2,16,11,14,6,10,13,\\ && 1,4,16; \\
$34$ & $13$ & 0,5,10,1,6,3,8,13,4,9,12,1,6,11,0,5,10,3,8,13,2,7,12,1,4,9,0,5,10,7,12,3,\\ && 8,13;\\
$35$ & $16$ & 0,3,6,9,2,5,15,11,0,16,13,8,2,14,7,3,0,16,5,1,11,14,7,10,13,8,5,12,3,6,16,\\ && 10,1,15,11;\\
$36$ & $13$ & (0,5,10,1,6,11,2,7,12,3,8,13)$^3$; \\
$37$ & $16$ & 0,3,6,1,4,7,10,16,12,9,2,13,4,7,14,11,6,0,10,5,16,9,14,2,12,7,1,11,16,0,3,\\ && 13,5,8,14,10,16; \\
$38$ & $13$ & (0,5,10,1,6,11,2,7,12,3,8,13)$^2$,0,5,10,1,6,11,2,7,12,3,8,13,4,9; \\
$39$ & $16$ & 0,3,6,9,2,5,8,11,0,7,10,13,2,15,12,1,16,7,10,4,14,9,0,13,16,6,3,11,5,2,10,\\ && 13,16,8,12,15,1,4,14, \\
$40$ & $13$ & 0,5,10,1,6,11,2,7,12,3,8,13,(0,5,10,1,6,11,2,7,12,3,8,13,4,9)$^2$; \\
$41$ & $16$ & (0,6,12,1,7,13,2,8,14,3,9,15)$^3$,4,10,16,5,11; \\
$42$ & $13$ & (0,5,10,1,6,11,2,7,12,3,8,13,4,9)$^3$; \\
$43$ & $15$ & 0,3,8,11,4,9,14,5,0,13,2,7,12,3,8,15,4,9,14,5,10,13,6,1,12,7,0,11,8,\\
&& 15,3,9,14,4,10,13,5,2,12,6,1,15,7; \\
$44$ & $13$ & (0,3,6,11,2,7,10,13,0,5,12,1,4,7,10,3,6,9,12,1,8,13)$^2$; \\
$45$ & $14$ & 0,5,10,1,6,11,2,7,12,3,8,13,4,9,14,5,10,0,6,11,1,7,12,2,8,13,3,9,14,\\
&& 4,10,0,5,11,1,6,12,2,7,13,3,8,14,4,9. \\
\hline
\end{tabular}
\normalsize

\end{document}